\pdfoutput=1
\documentclass[colorlinks,11pt,reqno]{amsart}

\usepackage{verbatim}
\usepackage{times}
\usepackage{stmaryrd}
\usepackage[a4paper]{geometry}
\usepackage{tabls}
\usepackage{url}
\usepackage[linktocpage = true]{hyperref}
\hypersetup{
 colorlinks = true,
 citecolor = blue,
}
\usepackage[draft]{fixme}
\fxsetup{layout=footnote, marginclue}
\usepackage{color}
\definecolor{red}{rgb}{1,0,0}
\definecolor{green}{rgb}{0,1,0}
\definecolor{blue}{rgb}{0,0,1}
\definecolor{refkey}{gray}{.625}
\definecolor{labelkey}{gray}{.625}

\usepackage[lite]{amsrefs}
\usepackage{amsfonts}
\usepackage{amssymb}
\usepackage{mathrsfs}
\usepackage{amsmath}
\usepackage{amsthm}
\usepackage{amscd}
\usepackage{enumerate}
\usepackage{paralist}
\usepackage{graphicx}
\usepackage{mathtools}
\usepackage{kantlipsum}
\usepackage{tikz-cd}
\allowdisplaybreaks

\DeclareMathOperator{\poly}{poly}

\DeclareMathOperator{\id}{id}
\DeclareMathOperator{\pr}{pr}
\DeclareMathOperator{\pbw}{\operatorname{pbw}}
\DeclareMathOperator{\Der}{Der}
\DeclareMathOperator{\Bott}{Bott}
\DeclareMathOperator{\Td}{Td}
\DeclareMathOperator{\tot}{tot}

\makeatletter
 \def\title@font{\normalsize\bfseries}
 \let\ltx@maketitle\@maketitle
 \def\@maketitle{\bgroup%
 \let\ltx@title\@title%
 \def\@\title{\resizebox{\textwidth}{!}{%
  \mbox{\title@font\ltx@title}%
 }}%
 \ltx@maketitle%
 \egroup}
\makeatother

\newcommand{\abs}[1]{\lvert#1\rvert}

\newcommand{\T}{\mathcal{T}}

\newcommand{\Z}{\mathbb{Z}}

\newcommand{\At}{\operatorname{At}}

\theoremstyle{plain}
\newtheorem{Thm}{Theorem}
\newtheorem{definition}{Definition}[section]
\newtheorem{corollary}[definition]{Corollary}
\newtheorem{theorem}[definition]{Theorem}
\newtheorem{lemma}[definition]{Lemma}
\newtheorem{proposition}[definition]{Proposition}
\theoremstyle{remark}
\newtheorem{remark}[definition]{Remark}

\begin{document}
\def\C{\mathbb{C}}
\def\CE{\mathrm{CE}}
\def\D{\mathcal{D}}
\def\Dol{\mathrm{Dol}}
\def\E{\mathcal{E}}
\def\F{\mathcal{F}}
\def\H{\mathbb{H}}
\def\k{\mathbb{K}}
\def\G{\mathcal{G}}
\def\M{\mathcal{M}}
\def\N{\mathcal{N}}
\def\m{\mathfrak{m}}
\def\O{\mathcal{O}}
\def\Q{\mathcal{Q}}
\def\P{\mathcal{P}}
\def\U{\mathcal{U}}
\def\L{\mathcal{L}}
\def\X{\mathcal{X}}
\def\Y{\mathcal{Y}}
\def\ZZ{\mathcal{Z}}
\def\Im{\operatorname{Im}}
\def\tot{\operatorname{tot}}
\def\End{\operatorname{End}}
\def\hkr{\operatorname{hkr}}
\def\Ber{\operatorname{Ber}}
\def\Hom{\operatorname{Hom}}
\def\Bott{\operatorname{Bott}}
\def\sgn{\operatorname{sgn}}
\def\rk{\operatorname{rank}}
\def\sh{\operatorname{sh}}
\newcommand{\OmegaF}{\Omega_F}
\newcommand\dF {d_F}
\newcommand{\ppx}[1]{\frac{\partial~}{\partial x^{#1}}}
\newcommand{\pullbackconn}{\nabla}
\newcommand{\sections}[1]{\Gamma(#1)}
\newcommand{\XX}{\mathfrak{X}}
\newcommand{\LdF}{L_{\dF}}
\newcommand{\TkM}{T_{\k} M}
\newcommand{\bas}{\mathrm{bas}}
\newcommand{\Ext}{\text{Ext}}
\newcommand{\cO}{\mathcal{O}}
\newcommand{\hochschildcohomology}[1]{{H\!H}^{#1}}
\newcommand{\xto}[1]{\xrightarrow{#1}}

\title{Hochschild cohomology of dg manifolds  associated to integrable distributions}

\author{Zhuo Chen}
\address{Department of Mathematics, Tsinghua University, Beijing, China.}
\email{\href{mailto:~~~chenzhuo@tsinghua.edu.cn}{chenzhuo@tsinghua.edu.cn}}

\author{Maosong Xiang}
\address{School of Mathematics and Statistics, Center for Mathematical Sciences, Huazhong University of Science and Technology, Wuhan, China.}
\email{\href{mailto:~~~msxiang@hust.edu.cn}{msxiang@hust.edu.cn}}

\author{Ping Xu}
\address{Department of Mathematics, Pennsylvania State University, State College, PA, USA.}
\email{\href{mailto:~~~ping@math.psu.edu}{ping@math.psu.edu}}
\thanks{Research partially supported by NSFC grant 12071241 (Chen), NSFC grant 11901221 (Xiang), and NSF grants DMS-1707545 and DMS-2001599 (Xu).}

\begin{abstract}
For the field $\k = \mathbb{R}$ or $\mathbb{C}$, and an integrable distribution $F \subseteq T_M \otimes_{\mathbb{R}} \k$ on a smooth manifold $M$, we study the Hochschild cohomology of the dg
manifold $(F[1],d_F)$ and establish a canonical isomorphism with the Hochschild cohomology of the algebra of functions on leaf space in terms of transversal polydifferential operators of $F$. In particular, for the dg manifold $(T_X^{0,1}[1],\bar{\partial})$ associated with a complex manifold $X$, we prove that its Hochschild cohomology is canonically isomorphic to the Hochschild cohomology $\hochschildcohomology{\bullet} (X)$ of the complex manifold $X$. As an application, we show that the Duflo-Kontsevich type theorem for the dg manifold $(T_X^{0,1}[1],\bar{\partial})$ implies the Duflo-Kontsevich theorem for complex manifolds.
\end{abstract}

\maketitle

\tableofcontents

\section{Introduction}

A \emph{dg manifold} is a pair $(\M,Q)$, where $\M$ is a $\Z$-graded manifold, and $Q$ is a homological vector field on $\M$, i.e., a degree $(+1)$ derivation of $C^\infty(\M)$ such that $[Q,Q] = 0$.
Homological vector fields first appeared in physics under the guise of BRST operators used to describe gauge symmetries. Since then, dg manifolds (a.k.a.\ $Q$-manifolds) have appeared frequently
in mathematical physics literature, e.g.,\ in the AKSZ formalism \cites{MR1432574,MR1230027}.
They also arise naturally in many situations in geometry, Lie theory, and mathematical physics.
To any complex manifold $X$ is associated a canonical dg manifold $\big(T^{0,1}_X[1], \bar{\partial}\big)$, where its algebra of functions is $C^\infty (T^{0,1}_X[1]) \cong \Omega^{0,\bullet}(X)$ and the homological vector field $Q$ is the  Dolbeault operator $\bar{\partial}$.
In this paper, by the Hochschild cohomology of a dg manifold $(\M,Q)$, following \cite{Tsygan},
we mean the {\it  direct sum} smooth Hochschild cohomology of the differential graded algebra $\big(C^\infty (\M) , Q\big)$. Alternatively, it can be defined as the cohomology $\H^\bullet\big(\tot(\D_{\poly}(\M)), \llbracket Q, - \rrbracket +d_{\mathscr{H}}\big)$ of the Hochschild cochain complex consisting of the {\it direct sum} polydifferential operators on $(\M, Q)$. See~\cites{CF, CFL,  CaTu, Kaledin, Keller, KL, Lunts, Ramadoss, RWang, Tsygan} and references  therein for  Hochschild cohomology in various situations.
Note that the direct sum  Hochschild cohomology of a differential graded algebra behaves significantly differently from the ordinary Hochschild cohomology, i.e., the direct product Hochschild cohomology \cite{CFL,  CaTu}.

One of the main goals of this paper is to compute the Hochschild cohomology groups of the dg manifold $\big(T^{0,1}_X[1], \bar{\partial}\big)$ by establishing a canonical isomorphism with the Hochschild cohomology groups $\hochschildcohomology{\bullet}(X)$ of the complex manifold $X$, which are defined as the groups $\Ext_{\cO_{X\times X}}^\bullet(\cO_{\Delta}, \cO_{\Delta})$~\cites{Caldararu, Markarian, Yekutieli}. The latter is known to be isomorphic to $\H^\bullet\big(\tot\big(\Omega_X^{0,\bullet}(\mathscr{D}_{\poly}(X))\big), \bar{\partial} + \id \otimes d_{\mathscr{H}}\big)$ \cite{Yekutieli} in terms of the Dolbeault resolution of the complex of sheaves
\[
 0 \rightarrow \O_X \rightarrow \mathscr{D}^1_{\poly}(X)\xrightarrow{d_{\mathscr{H}}} \mathscr{D}^2_{\poly}(X)  \xrightarrow{d_{\mathscr{H}}} \mathscr{D}^3_{\poly}(X) \rightarrow \cdots
\]
of holomorphic polydifferential operators over $X$.
As an application, applying the Duflo-Kontsevich type theorem for dg manifolds~\cites{LSX, SXsurvey}
to this particular dg manifold  $\big( T^{0,1}_X[1], \bar{\partial}\big)$, we recover the
well-known Duflo-Kontsevich theorem for complex manifolds~\cites{Kon, CV}.

For a given complex manifold $X$, $T_X^{0,1}\subset T_\C X := T_X \otimes \mathbb{C}$ is an integrable distribution. In this paper, we put this situation into a general framework by  considering general integrable distributions over the field $\k = \mathbb{R}$ or $\mathbb{C}$. In this way, we can include the case of dg manifolds associated to foliations, which should be of independent interest.
By an integrable distribution,  we mean a subbundle $F \subseteq {\TkM} = T_M \otimes_{\mathbb{R}} \k$, such that $\Gamma(F)$ is closed under the commutator of vector fields. When $\k = \mathbb{R}$, an integrable distribution $F$ is the tangent bundle of a regular foliation  on $M$ according to the Frobenius theorem. Meanwhile, each complex manifold $X$ determines an integrable distribution $F:= T_X^{0,1} \subset T_{\mathbb{C}} X$.
An integrable distribution $F \subseteq {\TkM}$ produces a finite dimensional dg manifold --- the leafwise de Rham differential, i.e., the Chevalley-Eilenberg differential of the Lie algebroid $F$, gives rise to a homological vector field $\dF$ on the graded manifold $F[1]$,  hence a dg manifold $(F[1], d_F)$.
For an integrable distribution $F$, the role of  holomorphic differential operators on a complex manifold $X$ is played by $F$-flat transversal differential operators $\D(T_\k M/F) := \frac{\D(M)}{\D(M)\Gamma(F)}$, and the role of Hochschild cohomology $\hochschildcohomology{\bullet}(X) \cong  \H^\bullet\big(\tot\big( \Omega_X^{0,\bullet} (\mathscr{D}_{\poly}(X)) \big), \bar{\partial} + \id \otimes d_{\mathscr{H}}\big)$ is played by the hypercohomology of
\[
\Big(\tot\big(\OmegaF\big(\D_{\poly}(T_\k M/F)\big)\big), d_F^\U + \id \otimes d_{\mathscr{H}} \Big),
\]
which can be thought of as the Hochschild cohomology of the algebra of functions on the leaf space of the foliation (in the case $\k = \mathbb{R}$). Here $\D_{\poly}(T_\k M/F)$ is the space of transversal polydifferential operators. See Section~\ref{Secpolyofpair} for details.
Our main result is to prove that there is a canonical isomorphism between
 $\H^\bullet\big(\tot(\D_{\poly}(F[1])), \llbracket d_F, - \rrbracket + d_{\mathscr{H}}\big)$, the Hochschild cohomology of the dg manifold $(F[1], d_F)$, and $\H^\bullet_{\CE} \big(F, (\D_{\poly}(T_\k M/F), d_{\mathscr{H}})\big)$.
To achieve this goal, we establish a homotopy contraction.
\begin{Thm}[Theorem~\ref{prop: contraction on Dpoly}]\label{Thm b}
Let $F \subseteq {\TkM}$ be an integrable distribution.
There is a contraction of dg $\Omega_F$-modules
        \begin{equation}\label{Eq:ContractDoplyF1}
        \begin{tikzcd}
       \Big (\tot\big(\D_{\poly}({F[1]})\big), \llbracket d_F, - \rrbracket + d_\mathscr{H}\Big ) \arrow[loop left, distance=2em, start anchor={[yshift=-1ex]west}, end anchor={[yshift=1ex]west}]{}{\breve{H}_\natural} \arrow[r,yshift = 0.7ex, "\Phi_\natural"] & \Big(\tot\big(\OmegaF(\D_{\poly}(T_\k M/F))\big) , \dF ^{\mathcal{U}} + \id \otimes d_\mathscr{H}\Big) \arrow[l,yshift = -0.7ex, "\breve{\Psi}_\natural"].
        \end{tikzcd}
        \end{equation}
\end{Thm}
The construction of such a contraction is highly nontrivial. To do so, following Vitagliano~\cite{Luca}, we first establish a contraction from the left $\Omega_F$-module $\D({F[1]})$ of differential operators on  $(F[1],d_F)$ onto the space $\Omega_F(\D(T_\k M/F))$ of transversal differential operators of $F$ (Theorem~\ref{thm: Luca's result}).
Applying the tensor trick (cf.~\cites{Manetti, Manettibook}) to this contraction and using the perturbation lemma, we obtain  the desired contraction \eqref{Eq:ContractDoplyF1}.
Although the construction of the contraction involves of choices of certain geometric data such as connections and splittings,  the projection $\Phi_\natural$ is independent of those choices and is
canonical. Therefore, the induced isomorphism on the level of cohomology groups is indeed canonical:
\begin{equation}
\label{eq:Phi}
\Phi_\natural \colon \H^\bullet \big(\tot \big(\D_{\poly}(F[1])\big), \llbracket d_F, - \rrbracket + d_{\mathscr{H}}\big) \xto{\cong}
\H^\bullet_{\CE}\big(F, \big(\D_{\poly}(T_\k M/F),d_{\mathscr{H}}\big)\big).
\end{equation}
For polyvector fields,  it was already proved in~\cites{BCSX, CXX} that there exists an isomorphism of Gerstenhaber algebras
\begin{equation}\label{Eqt:isoHTpolyFoneTpolyB}
        \Phi \colon \H^\bullet\big(\tot\big(\mathcal{T}_{\poly}(F[1])\big), L_{d_F}\big) \xrightarrow{\cong}
        \H_{\CE}^\bullet\big(F, \mathcal{T}_{\poly}(T_\k M/F)\big).
\end{equation}
Here $\H^\bullet\big(\tot(\mathcal{T}_{\poly}(F[1])), L_{d_F}\big)$ denotes the hypercohomology of polyvector fields  of the dg manifold $(F[1], d_F)$, while $\H_{\CE}^\bullet\big(F, \mathcal{T}_{\poly}(T_\k M/F)\big)$ denotes the hypercohomology of
\[
\big(\tot(\OmegaF (\T_{\poly}(T_\k M/F))),d_F\big),
\]
which, in the case $\k = \mathbb{R}$,  can be thought of as the leafwise de Rham cohomology with coefficients in transversal polyvector fields of the foliation.

The two isomorphisms $\Phi_\natural$ and $\Phi$ are in fact compatible in the following way.
\begin{Thm}[Theorem~\ref{Thm: KD for integrable distributions}]\label{Main thm}
        We have the following commutative diagram:
        \[
        \begin{tikzcd}
                \H^\bullet\Big(\tot\big(\mathcal{T}_{\poly}(F[1])\big), L_{\dF}\Big) \ar{d}[left]{\Phi}[right]{\cong}
                \ar{rrr}{~\hkr \circ \Td^{1/2}_{(F[1],d_F)}~} &&& \H^\bullet\Big(\tot \big(\D_{\poly}(F[1])\big), \llbracket d_F, - \rrbracket + d_{\mathscr{H}}\Big) \ar{d}[right]{\Phi_\natural}[left]{\cong} \\
                \H^\bullet_{\CE}\Big(F,\mathcal{T}_{\poly}(T_\k M/F)\Big)
                \ar{rrr}{~\hkr \circ \Td^{1/2}_{{\TkM}/F}~} &&& \H^\bullet_{\CE}\Big(F, \big(\D_{\poly}(T_\k M/F),d_{\mathscr{H}}\big)\Big).
        \end{tikzcd}
        \]
\end{Thm}
Here $\Td_{(F[1],d_F)}$ and $\Td_{{\TkM}/F}$ are the Todd classes of the dg manifold $(F[1],d_F)$ and the Lie pair $(\TkM, F)$, respectively, which act by contractions. By abuse of notation, $\hkr$ stands for the  Hochschild-Kostant-Rosenberg maps for both the dg manifold $(F[1],d_F)$ and  the Lie pair $(\TkM, F)$.

The map $\Phi_\natural$ intertwines the associative products on $\tot\left(\D_{\poly}({F[1]})\right)$ and $\tot\left(\OmegaF ( \D_{\poly}(T_\k M/F))\right)$ (see Theorem~\ref{prop: contraction on Dpoly}).
When $F$ is \emph{perfect}, that is, if there exists another integrable
 distribution which is transversal to $F$,  we further prove that $\Phi_\natural$ in \eqref{eq:Phi}
is an isomorphism of Gerstenhaber algebras (see Theorem~\ref{prop: Isom of G-algebra of Dpoly for perfect ID}).

As an application of Theorem~\ref{Main thm}, we consider complex manifolds.
For a complex manifold $X$, $F = T_X^{0,1} \subset T_\C X$ is a perfect integrable distribution, since $T_\C X  = T_X^{1,0} \bowtie T_X^{0,1}$ is a matched pair of Lie algebroids, that is,
$T_X^{1,0}$ is an integrable distribution transversal to $T_X^{0,1}$.
Thus the quotient bundle $T_\C X/T_X^{0,1}$ is naturally identified with $ T^{1,0}_X$.

Based on the discussions above and by applying Theorem~\ref{Main thm}, we establish the following
\begin{Thm}[Theorem~\ref{Thm: Complex manifolds}]\label{Thm C}
  Let $(T_X^{0,1}[1],\bar{\partial})$ be the dg manifold arising from a complex manifold $X$. Then we have the following commutative diagram
  \[
   \begin{tikzcd}
   \H^\bullet\Big(\tot\big(\mathcal{T}_{\poly}(T_X^{0,1}[1])\big), L_{\bar{\partial}} \Big)   \ar{d}[left]{\Phi}[right]{\cong}
  \ar{rrr}{\hkr \circ \Td^{1/2}_{ (T_X^{0,1}[1],\bar{\partial})}} &&& \H^\bullet\Big(\tot\big(\D_{\poly}(T_X^{0,1}[1])\big), L_{\bar{\partial}}  + d_{\mathscr{H}}\Big) \ar{d}[right]{\Phi_\natural}[left]{\cong} \\
   \H^\bullet \big(X, \mathcal{T}_{\poly}(X)\big) \ar{rrr}{\hkr \circ \Td^{1/2}_{T_\C X/T_X^{0,1}}} &&& \hochschildcohomology{\bullet}(X),
   \end{tikzcd}
  \]
where both $\Phi$ and $\Phi_{\natural}$ are isomorphisms of Gerstenhaber algebras.
\end{Thm}
As an immediate  consequence, applying the Duflo-Kontsevich type theorem for the dg manifold
 $(T_X^{0,1}[1], \bar{\partial})$ \cites{LSX, SXsurvey},
 we recover Duflo-Kontsevich theorem for complex manifolds, a theorem first
 proved by Kontsevich (for associative algebras only) in~\cite{Kon}, Calaque and Van den Bergh in~\cite{CV} and recovered by Liao, Sti\'{e}non and Xu using Lie pairs in~\cite{LSXpair}.

\begin{Thm}[Theorem~\ref{main corollary}]\label{Thm D}
  For every complex manifold $X$, the composition
  \[
    \hkr \circ \Td^{1/2}_{T_\C X/T_X^{0,1}} \colon
 \H^\bullet\big(X, \mathcal{T}_{\poly}(X) \big) \xto{\cong} \hochschildcohomology{\bullet}(X)
  \]
is an isomorphism of Gerstenhaber algebras.
\end{Thm}

Finally, we would like to point out that without the perfect assumption on $F$, the hypercohomology $\H_{\CE}^\bullet\big(F, $ $(\D_{\poly}(T_\k M/F), d_{\mathscr{H}})\big)$ still carries a canonical Gerstenhaber algebra structure---  a result due to
Bandiera, Sti\'{e}non and Xu ~\cite{BSX}.  We expect that $\Phi_\natural$  in \eqref{eq:Phi} is still
an isomorphism of Gerstenhaber algebras. However, the construction of the Gerstenhaber bracket in~\cite{BSX} is not explicit (see Remark~\ref{Main remark}).
Our method here cannot be applied  directly  to prove that  $\Phi_\natural$  in \eqref{eq:Phi} respects the Gerstenhaber brackets.
We wish to return to this question in the future.

\paragraph{\textbf{Acknowledgement}}
We would like to thank Ruggero Bandiera, Hsuan-Yi Liao, Seokbong Seol, Mathieu Sti\'enon,
 Luca Vitagliano and Zhengfang Wang for fruitful discussions and useful comments.
We are also grateful to the anonymous referee for constructive suggestions to improve the presentation of the manuscript.

\section{Algebraic structures of differential operators on the dg manifold $(F[1],d_F)$}\label{Sec: Diff operators}
Let $F \subseteq {\TkM}$ be an integrable distribution and denote by $R$ the algebra of $\k$-valued smooth functions $C^\infty(M,\k)$ on $M$. Consider the graded manifold $F[1]$ whose algebra of smooth functions $C^\infty({F[1]},\k)$ is identified with $\OmegaF:= \Gamma(\wedge F^\vee)$. The leafwise de Rham differential, i.e., the Chevalley-Eilenberg differential $d_F \colon \Omega_F^\bullet \to \Omega_F^{\bullet+1}$ of the Lie algebroid $F$, can be viewed as a homological vector field on the graded manifold $F[1]$. Thus we obtain a dg manifold $(F[1], d_F)$.
For the Chevalley-Eilenberg complex of a Lie algebroid, see \cite{Macbook}*{Section 7.1}.

\subsection{Two dg coalgebras of differential operators}
To any integrable distribution $F$ are associated two dg coalgebras. The first one is the space $\D(F[1])$ of differential operators on the dg manifold $(F[1], d_F)$. Note that the homological vector field $d_F$ belongs to $\Gamma(T_{F[1]}) \subset \D(F[1])$, thus induces a degree $(+1)$ differential:
\[
\llbracket d_F, - \rrbracket \colon \D(F[1]) \to \D(F[1])[1].
\]
Here $\llbracket -,- \rrbracket$ denotes the graded commutator on $\D(F[1])$. The differential $\llbracket d_F,- \rrbracket$ preserves the natural increasing filtration on $\D(F[1])$ by the order of differential operators
\begin{equation}\label{Eq: filtration on DF[1]}
\Omega_F \cong \D^{\leq 0}(F[1]) \subset \D^{\leq 1}(F[1]) \cong \Omega_F \oplus \Gamma(T_{F[1]}) \subset \cdots \subset \D^{\leq k}(F[1]) \subset \D^{\leq k+1}(F[1]) \subset \cdots.
\end{equation}
Moreover, $\D(F[1])$ admits a natural $\Omega_F$-coproduct
\begin{equation}\label{Eqt:DeltaDFone}
\Delta \colon \D(F[1]) \to \D(F[1]) \otimes_{\Omega_F} \D(F[1])
\end{equation}
such that
\[
\Delta(D) (\xi \otimes \eta) = D(\xi \wedge \eta),
\]
for any $D \in \D(F[1])$ and any $\xi, \eta \in \Omega_F$. It can be verified directly that the differential $\llbracket d_F,- \rrbracket$ is a coderivation with respect to this coproduct $\Delta$. Thus, the triple $(\D(F[1]), \llbracket d_F,- \rrbracket, \Delta)$ forms a filtered dg coalgebra over the dg algebra $(\Omega_F, d_F)$.

The second dg coalgebra arises from the Lie pair $({\TkM}, F)$ and the natural Lie algebroid $F$-module structure on the normal bundle $B:= {\TkM}/F$, which is known as the Bott connection~\cite{CSX}, defined by
\[
  \nabla^{\operatorname{Bott}}_a b=\pr_B [a, u],
\]
for any $a \in \Gamma(F)$, $b \in \Gamma(B)$ and $u \in \Gamma({\TkM})$ such that $\pr_B(u) = b$. Here $\pr_B \colon T_\k M \to B$ is the canonical projection.
Consider the space $\D(M)$ of $\k$-linear differential operators on $M$. When viewed as a filtered $R$-coalgebra, $\D(M)$ is indeed the universal enveloping algebra of the Lie algebroid ${\TkM}$  (cf.~\cite{Xu}). We will also use the same symbol $\Delta$, by abuse of notation, to denote the standard coproduct on $\D(M)$:
\begin{equation}\label{Eq: coproduct on DM}
\Delta \colon\ ~ \D(M) \to \D(M) \otimes_{R} \D(M)
\end{equation}
Let $\D(M)\Gamma(F) \subseteq \D(M)$ be the left ideal of $\D(M)$ generated by $\Gamma(F)$. Since
\[
 \Delta(\D(M)\Gamma(F)) \subseteq \D(M) \otimes_R \D(M)\Gamma(F) + \D(M)\Gamma(F) \otimes_R \D(M),
\]
the quotient space
\[
\D(B):=\frac{\D(M)}{\D(M)\Gamma(F)}
\]
inherits a coproduct structure
\begin{equation}\label{Eq: coproduct on DB}
  \Delta \colon\ ~  \D(B) \to \D(B) \otimes_{R} \D(B),
\end{equation}
from the coproduct $\Delta$~\eqref{Eq: coproduct on DM} on $\D(M)$. Thus $\D(B)$ is also an $R$-coalgebra, which we call the $R$-coalgebra of \textit{transversal differential operators} of the integrable distribution $F$~\cite{Luca}.
Moreover, the natural filtration on $\D(M)$ determined by the order of differential operators descends to a filtration on the $R$-coalgebra $\D(B)$
\begin{equation}\label{Eq: filtration on DB}
 R \subset \D^{\leq 1} (B) \cong R \oplus \Gamma(B) \subset \cdots \subset \D^{\leq k}(B) \subset \D^{\leq k+1}(B) \subset \cdots.
\end{equation}
Note that in general $\D(B)$ is not an associative algebra. According to Vitagliano~\cite{Luca}, there is an $A_\infty$ algebra structure on $\Omega_F(\D(B)):= \Omega_F \otimes_R \D(B)$.

The $R$-coalgebra $\D(B)$ admits a canonical $F$-module structure defined by
\[
   a \cdot \overline{u} = \overline{a \circ u},
\]
for any $a \in \Gamma(F)$ and $\overline{u} \in \D(B)$ that is the projection of $u \in \D(M)$.
Here $\circ$ denotes the composition of differential operators.
Denote by $d_F^\U$ its associated Chevalley-Eilenberg differential on $\Omega_F(\D(B))$.
In order to obtain an explicit formula for the differential $d_F^\U$, we consider the bundle projection $\pi \colon F[1] \to M$. Its tangent map
\[
\pi_\ast \colon T_{F[1]} \to \pi^\ast {\TkM}
\]
is a map of vector bundles over the graded manifold $F[1]$. Consider the image $\pi_\ast(d_F)$ of the Chevalley-Eilenberg differential $d_F \in \Gamma(T_{F[1]})$.
Locally one can always write
\begin{equation}\label{Eq:piastdF}
\pi_\ast(d_F)  = \sum_i \alpha^i \otimes u_i \in \Gamma(F[1]^\vee \otimes F) \cong \Omega^1_F \otimes_R \Gamma(F) \subseteq \Gamma(\pi^\ast {\TkM}),
\end{equation}
where $\{u_i\}$ is any local frame of $F$ and $\{\alpha^i\}$ is its dual local frame of $F[1]^\vee$.
Therefore, we have
\begin{align}\label{Eq: dFU}
   d_F^\U(\xi \otimes \overline{u}) &= d_F(\xi) \otimes \overline{u} + (-1)^{\abs{\xi}}\sum_i \xi \wedge \alpha^i \otimes \overline{u_i \circ u} \notag \\
   &= d_F(\xi) \otimes \overline{u} + (-1)^{\abs{\xi}} \xi \wedge \overline{\pi_\ast(d_F) \circ u},
\end{align}
for any homogeneous $\xi \in \Omega_F$ and $\overline{u} \in \D(B)$.
From this, we can verify that $d_F^\U$ preserves the filtration~\eqref{Eq: filtration on DF[1]}, and moreover it is a coderivation with respect to the $\Omega_F$-linear coproduct
\[
\Delta \colon \Omega_F(\D(B)) \to \Omega_F(\D(B)) \otimes_{\Omega_F} \Omega_F(\D(B)),
\]
which is a natural extension of the coproduct~\eqref{Eq: coproduct on DB} on $\D(B)$. Thus $(\Omega_F(\D(B)), d_F^\U, \Delta)$ is a filtered dg coalgebra over $(\Omega_F, d_F)$.

The key fact is the following
\begin{theorem}\label{thm: Luca's result}
There exists a filtered contraction of dg $\Omega_F$-modules
	\begin{equation}\label{Contration:DFonetoOmegaFDpolyB}
	\begin{tikzcd}
	\big(\D({F[1]}), \llbracket d_F, - \rrbracket\big) \arrow[loop left, distance=2em, start anchor={[yshift=-1ex]west}, end anchor={[yshift=1ex]west}]{}{H_\natural} \arrow[r, yshift = 0.7ex, "\Phi_\natural"] & \big(\Omega_F(\D(B)) ,  d_F^{\U}\big) \arrow[l, yshift = -0.7ex, "\Psi_\natural"],
	\end{tikzcd}
	\end{equation}
where the projection $\Phi_\natural$ is a morphism of $\Omega_F$-coalgebras.
\end{theorem}

In general, the inclusion $\Psi_\natural$ and the homotopy map $H_\natural$ are \textit{not} morphisms of coalgebras. The construction of this contraction is due to Vitagliano~\cite{Luca}. However, the coalgebra structure was not addressed and many details of verification were skipped in~\cite{Luca}.
For completeness, we will follow Vitagliano's construction to give a thorough proof of Theorem~\ref{thm: Luca's result} in the subsequent subsection.

We call an integrable distribution $F \subseteq \TkM$ \emph{perfect}  if
there exists a transversal integrable distribution $B \subseteq \TkM$.  In this case, $F \bowtie B$ forms a matched pair of Lie algebroids~\cites{MM, Mokri}.
In this paper, perfect integrable distributions are of particular interest to us since integrable distributions arising from complex manifolds are perfect. For a perfect integrable distribution, the space of transversal differential operators $\D(B)$ can be naturally identified with the space $\U(B)$ of the universal enveloping algebra of the Lie algebroid $B$, which is a Hopf algebroid~\cite{Xu}.
Furthermore, it is proved by Bandiera, Sti\'{e}non and Xu that the complex $\big(\Omega_F(\U(B)), d_F^\U\big)$ admits a structure of dg Hopf algebroid over the commutative dg algebra $(\Omega_F, d_F)$~\cite{BSX}:
\begin{compactenum}
  \item The associative multiplication ``$\cdot$'' on $\Omega_F(\U(B))$ is defined by the relation
  \begin{equation}\label{Eq: product}
   (\xi \otimes b_1b_2\cdots b_n) \cdot (\eta \otimes u) = \sum_{k=0}^{n}\sum_{\sigma \in \sh(k,n-k)} (\xi \wedge \eth_{b_{\sigma(1)}}\cdots \eth_{b_{\sigma(k)}} \eta) \otimes b_{\sigma(k+1)}\cdots b_{\sigma(n)}\cdot u,
  \end{equation}
  for any $\xi, \eta \in \Omega_F, b_1,\cdots, b_n \in \Gamma(B)$, and $u \in \U(B)$, where $\eth$ is the $B$-connection on $\wedge F^\vee$ induced from the Bott-$B$-connection $\eth$ on $F$ defined by $\eth_b a = \pr_F[b,a]$ for any $b \in \Gamma(B), a \in \Gamma(F)$. Here $\pr_F \colon T_\k M \to F$ is the canonical projection.
  \item The source and target maps $\alpha,\beta \colon\ ~ \Omega_F \to \Omega_F(\U(B))$ are both inclusions.
  \item The comultiplication $\Delta$ is the $\Omega_F$-linear extension of that of the Hopf algebroid $\U(B)$~\cite{Xu}.
\end{compactenum}

\begin{theorem}\label{thm:contraction-on-DM}
Assume that $F \subseteq {\TkM}$ is a perfect integrable distribution with a transverse integrable distribution $B \subseteq {\TkM} $. Then the contraction~\eqref{Contration:DFonetoOmegaFDpolyB},  which now reads
 \[
	\begin{tikzcd}
	\big(\D({F[1]}), \llbracket d_F, - \rrbracket\big) \arrow[loop left, distance=2em, start anchor={[yshift=-1ex]west}, end anchor={[yshift=1ex]west}]{}{H_\natural} \arrow[r,yshift = 0.7ex, "\Phi_\natural"] & \big(\OmegaF(\U(B)) ,  \dF^{\mathcal{U}}\big),  \arrow[l,yshift = -0.7ex, "\Psi_\natural"]
	\end{tikzcd}
\]
can be chosen so that the inclusion $\Psi_\natural$ is a morphism of dg Hopf algebroids over the commutative dg algebra $(\Omega_F, d_F)$. That is, $\Psi_\natural$ is compatible with the source and target maps, multiplications, and comultiplications in the sense that
  \begin{align*}
   \Psi_\natural\big((\xi \otimes u) \cdot (\xi^\prime \otimes u^\prime)\big)  = \Psi_\natural(\xi \otimes u) \cdot \Psi_\natural(\xi^\prime \otimes u^\prime)  & \qquad\mbox{~and~}\quad &  \Psi_\natural\big(\Delta(\xi \otimes u)\big)  = \Delta \big(\Psi_\natural(\xi \otimes u)\big),
  \end{align*}
  for any $\xi, \xi^\prime \in \Omega_F$ and $u, u^\prime \in \U(B)$. Here, by abuse of notation, we use the same symbol $\Psi_\natural$ to denote its extension to $\OmegaF(\U(B)) \otimes_{\OmegaF} \OmegaF(\U(B)) \to \D({F[1]}) \otimes_{\OmegaF} \D({F[1]})$.
\end{theorem}
However,  the projection $\Phi_\natural$, being a morphism of $\Omega_F$-coalgebras, is not necessarily compatible with the multiplications; thus it is not a morphism of dg Hopf algebroids.

\subsection{Proof of Theorem~\ref{thm: Luca's result}} 	
We mainly follow Vitagliano's approach for the construction of contraction data.
The first step is to construct a filtered contraction for the dg module $(\Gamma(ST_{F[1]}), L_{d_F})$ of symmetric contravariant tensor fields on the dg manifold $(F[1], d_F)$ over $(\Omega_F(SB), d_F^{SB})$ (see Proposition~\ref{prop: coalgebra contraction on STF}).
Then we need to take a detour via two types of Poincar\'{e}-Birkhoff-Witt isomorphisms --- one is $\pbw\colon \Gamma(ST_{F[1]}) \rightarrow \D({F[1]})$ for the graded manifold $F[1]$~\cite{LS}, and the other is $\overline{\pbw}\colon \Gamma(SB) \to \D(B)$ for the Lie pair $({\TkM}, F)$~\cite{LGSX}.
These two PBW isomorphisms are crucial to the proof of Theorem~\ref{thm: Luca's result}. They allow us to construct the desired contraction~\eqref{Contration:DFonetoOmegaFDpolyB}  by transferring the problem to  that of $\Gamma(ST_{F[1]})$ onto $\OmegaF(SB)$ with non-standard differentials, for which we use homological perturbation lemma.
Note that Vitagliano's construction in~\cite{Luca} also relies on two PBW maps which he denoted by $\operatorname{PBW} \colon \Gamma(ST_{F[1]}) \to \D(F[1])$ and $\underline{\operatorname{PBW}} \colon \Omega_F(SB) \to \Omega_F(\D(B))$, all defined by local charts.
It is not hard to check that they coincide with ours using the iteration formulas.

\subsubsection{A contraction for symmetric contravariant tensor fields on the graded manifold $F[1]$}
Let $F \subseteq {\TkM}$ be an integrable distribution with normal bundle $B = {\TkM}/F$. There is a short exact sequence of vector bundles over $M$:
\begin{equation}\label{SES}
0 \rightarrow F \xrightarrow{i} {\TkM} \xrightarrow{\pr_B} B \rightarrow 0.
\end{equation}
Consider the space $\Gamma\left(ST_{F[1]}\right) = \oplus_{k \geq 0} \Gamma\left(S^kT_{F[1]}\right)$ of symmetric contravariant tensor fields on the graded manifold $F[1]$, i.e., sections of symmetric tensor products of the tangent bundle $T_{F[1]}$.
Note that $\Gamma\left(ST_{F[1]}\right)$ is an $\Omega_F$-coalgebra and carries an increasing filtration bounded below:
\begin{align*}
\Omega_F &\cong \Gamma\left(S^{\leq 0}T_{F[1]}\right)  \subseteq \cdots \subseteq \Gamma\left(S^{\leq k}T_{F[1]}\right)  \subseteq  \Gamma\left(S^{\leq k+1}T_{F[1]}\right)  \subseteq \cdots.
\end{align*}
The $\Omega_F$-coalgebra $\OmegaF\left(SB\right) = \oplus_{k\geq 0}\Omega_F\left(S^k B\right)$ also admits an increasing filtration bounded below:
\begin{align*}
 \Omega_F &\cong \OmegaF\left(S^{\leq 0} B\right) \subseteq \cdots \subseteq \OmegaF\left(S^{\leq k} B\right) \subseteq \OmegaF\left(S^{\leq k+1} B\right) \subseteq \cdots.
\end{align*}
The first key result is the following
\begin{proposition}\label{prop: coalgebra contraction on STF}
For each splitting $j$ of the short exact sequence~\eqref{SES} and a torsion-free  $F$-connection $\widetilde{\nabla}^F$ on $F$, there is a filtered contraction
\begin{equation}\label{Eq: contraction in CXX}
	\begin{tikzcd}
	\big(\Gamma(ST_{F[1]}), \LdF\big) \arrow[loop left, distance=2em, start anchor={[yshift=-1ex]west}, end anchor={[yshift=1ex]west}]{}{H} \arrow[r,yshift = 0.7ex, "\Phi"] & \big(\OmegaF(SB) , d^{SB}_F\big) \arrow[l,yshift = -0.7ex, "\Psi"],
	\end{tikzcd}
\end{equation}
 satisfying the condition that both $\Phi$ and $\Psi$ are morphisms of $\Omega_F$-coalgebras.
\end{proposition}
\begin{remark}
 However, the contraction~\eqref{Eq: contraction in CXX} is not a coalgebra contraction in the sense of~\cite{Manetti}.
\end{remark}
This is a direct consequence of~\cite{CXX}*{Proposition 2.17} (see also~\cite{AC}). Below we give a more conceptual proof.

\emph{Step 1 -- An explicit description of complexes of contravariant tensor fields on $(F[1],d_F)$.}
According to \cite{GSM} (see also~\cite{Mehta}), for each $k \geq 1$, the complex $\big(\Gamma(S^kT_{F[1]}), \LdF\big)$ can be identified as a representation up to homotopy of the Lie algebroid $F$ on the graded vector bundle $S^k(F[1] \oplus T_\k M)$ over $M$. We recall its construction briefly below.

Observe that there is a short exact sequence of vector bundles over the graded manifold ${F[1]}$:
\begin{equation}\label{SES of tangent bundle on F[1]}
\begin{tikzcd}
	0 \ar{r} & \pi^\ast(F[1]) \ar{r}{I} & T_{F[1]} \ar{r}{\pi_\ast} & \pi^\ast({\TkM}) \ar{r} & 0,
\end{tikzcd}
\end{equation}
where $\pi_\ast\colon T_{F[1]} \rightarrow \pi^\ast({\TkM})$ is the tangent map of the bundle projection $\pi \colon F[1] \rightarrow M$, and $I$ is the canonical vertical lifting.
Taking global sections gives rise to a short exact sequence of left $\Omega_F$-modules:
\[
\begin{tikzcd}
	0 \ar{r} & \OmegaF \otimes_R \Gamma(F[1]) \ar{r}{I} & \Gamma(T_{F[1]}) \ar{r}{\pi_\ast} & \OmegaF \otimes_R \Gamma(T_\k M)   \ar{r} & 0.
\end{tikzcd}
\]
Here the canonical vertical lifting is the $\Omega_F$-linear contraction: for any $\omega \in \OmegaF$ and $a[1] \in \Gamma(F[1])$,
\begin{equation}\label{Eq: Def of I}
I(\omega \otimes a[1]) = \omega \otimes \iota_a.
\end{equation}
Let us choose a linear connection $\nabla^F$ on the vector bundle $F$ over $M$.
This connection $\nabla^F$ induces a splitting of the short exact sequence~\eqref{SES of tangent bundle on F[1]}. Then $T_{F[1]}$ can be identified with $F[1] \times_M (F[1] \oplus T_\k M)$. Thus, one has an isomorphism of $\OmegaF$-modules
\begin{align}\label{Eqt:splitXFone}
 \Gamma\left(T_{F[1]}\right) &\stackrel{\cong}{\longrightarrow} \Omega_F (F[1] \oplus T_\k M) = \Omega_F \otimes_R \Gamma(F[1] \oplus T_\k M),
\end{align}
The isomorphism~\eqref{Eqt:splitXFone} transfers the Lie derivative $\LdF$ on $\Gamma(T_{F[1]})$ to a square zero derivation
\begin{equation}\label{Eq: Def of DnablaF}
	D_{\nabla^F}=\delta+d_{\nabla^{\bas}}+{R^{\bas}_{\nabla^F}}\colon \Omega_F (F[1] \oplus T_\k M) \to \Omega_F (F[1] \oplus T_\k M)[1],
\end{equation}
where
\begin{itemize}
	\item $\delta$ is an $\Omega_F$-linear derivation determined by
      \begin{equation}\label{Eq: Def of delta}
        \delta \colon \Gamma(F[1]\oplus T_\k M) \to \Gamma(F[1] \oplus T_\k M)[1], \qquad  \delta(a[1] + u) = a,
      \end{equation}
      for any $a[1] \in \Gamma(F[1])$ and $u \in \Gamma(T_\k M)$;
	\item $d_{\nabla^{\bas}} \colon \OmegaF^\bullet(F[1]\oplus T_\k M) \to \Omega^{\bullet+1}_F (F[1] \oplus T_\k M)$ is the covariant derivative of the basic $F$-connection $\nabla^{\bas}$ on $F[1] \oplus T_\k M$ defined by
   \begin{equation}\label{Eq: def of basic connection on TM}
     \nabla^{\bas}_a(u):=\nabla^F_u a+[a,u],
   \end{equation}
   and
   \begin{equation}\label{Eq: def of basic connection on F[1]}
    \nabla^{\bas}_a(a^\prime[1]):= (\nabla^{\bas}_a a^\prime)[1] = (\nabla^F_{a^\prime} a+[a,a^\prime])[1],
   \end{equation}
  for any $a, a^\prime \in \Gamma(F)$ and $u \in \Gamma(\TkM)$;
	\item ${R^{\bas}_{\nabla^F}} \in \Omega_F^2(\Hom(T_\k M, F[1]))$, known as the basic curvature of $\nabla^F$, defines an $\Omega_F$-linear map $\Omega^\bullet_F(T_\k M) \to \Omega_F^{\bullet+2}(F[1])$ by
	 	\begin{align*}
 {R^{\bas}_{\nabla^F}}(u)(a^\prime, a^{\prime\prime}) &   :=\left(\nabla^F_{u}[a^\prime, a^{\prime\prime}]-[\nabla^F_u a^\prime, a^{\prime\prime}]-[a^\prime, \nabla^F_u  a^{\prime\prime}] - \nabla^F_{\nabla^{\bas}_{a^{\prime\prime}}u}a^\prime +\nabla^F_{\nabla^{\bas}_{a^\prime} u} a^{\prime\prime}\right)[1],
	 	\end{align*}
for any $a^\prime, a^{\prime\prime} \in \Gamma(F)$ and $u\in \Gamma(\TkM)$.
	 \end{itemize}
The pair $\big(F[1] \oplus T_\k M, D_{\nabla^F}\big)$ is a homotopy $F$-module or a representation up to homotopy of $F$. For details, see~\cites{AC, GSM}.

The isomorphism~\eqref{Eqt:splitXFone} extends, by taking symmetric tensor product, to contravariant tensor fields on $(F[1],d_F)$
 \begin{align}\label{Eqt:splitXFone2}
 \Gamma\left(S^kT_{F[1]}\right) &\stackrel{\cong}{\longrightarrow} \Omega_F \left(S^k(F[1] \oplus T_\k M)\right).
 \end{align}
Meanwhile, the differential $D_{\nabla^F}$~\eqref{Eq: Def of DnablaF} on $\Omega_F(F[1]\oplus T_\k M)$ extends by Leibniz rule to a differential on $\Omega_F\big(S^k(F[1] \oplus T_\k M)\big)$, which is still denoted by $D_{\nabla^F}$ by abuse of notation.
The isomorphism~\eqref{Eqt:splitXFone2} becomes an isomorphism of cochain complexes
\[
  	\Big(\Gamma(S^kT_{F[1]}), \LdF\Big) \xrightarrow{\cong} \Big(\Omega_F\big(S^k(F[1] \oplus T_\k M)\big), D_{\nabla^F} = \delta + d_{\nabla^{\bas}} + {R^{\bas}_{\nabla^F}}\Big).
\]

The differential $D_{\nabla^F}$~\eqref{Eq: Def of DnablaF} can be simplified if we choose a special linear connection on $F$.
Choose a torsion-free $F$-connection $\widetilde{\nabla}^F$ on $F$ and a splitting $j$ of the short exact sequence~\eqref{SES}. There is an induced linear connection on $F$ defined as follows.
The projection $\pr_F \colon \TkM \to F$   determines a Bott ``$B$-connection"~\cite{CSX} on $F$~\footnote{When $F \subseteq T_\k M$ is perfect, i.e., $B$ is a Lie algebroid, $\eth$ defined by Equation~\eqref{Eq: Bott B-connection} becomes the genuine Bott $B$-connection on $F$.}:
 \begin{equation}\label{Eq: Bott B-connection}
 \eth \colon \Gamma(B)\otimes \Gamma(F)\to \Gamma(F),  \quad (b,a) \mapsto \eth_b a = \pr_F [j(b),a].
 \end{equation}
Define
 \begin{equation}\label{Eq: nablaF}
 \nabla^F_u a = \widetilde{\nabla}^F_{\pr_F(u)}a + \eth_{\pr_B(u)}a =  \widetilde{\nabla}^F_{\pr_F(u)}a + \pr_F [\pr_B(u),a],
 \end{equation}
 for any $u \in \Gamma({\TkM})$ and $a \in \Gamma(F)$. It is easy to see that $\nabla^F$ defined above is indeed a linear connection on $F$.

\begin{lemma}\label{lem: special connection on F}
The basic curvature ${R^{\bas}_{\nabla^F}} \in \Omega^2_F(\Hom(T_\k M, F[1]))$ of the linear connection $\nabla^F$ defined in~\eqref{Eq: nablaF} satisfies
 	\begin{align*}
 	{R^{\bas}_{\nabla^F}}(a,a^\prime)(a^{\prime\prime}) = -R_{\widetilde{\nabla}^F}(a,a^\prime)a^{\prime\prime}, &\qquad \mbox{ ~and}   & {R^{\bas}_{\nabla^F}}(a,a^\prime)(j(b)) = 0,
 	\end{align*}
for any $a, a^{\prime}, a^{\prime\prime} \in \Gamma(F)$, $b \in \Gamma(B)$, where $R_{\widetilde{\nabla}^F} \in \Omega_F^2(\End F)$ denotes the curvature of $\widetilde{\nabla}^F$.
 \end{lemma}
 \begin{proof}
Since $\widetilde{\nabla}^F$ is torsion-free, the associated basic $F$-connection $\nabla^{\bas}$ as in~\eqref{Eq: def of basic connection on TM} and~\eqref{Eq: def of basic connection on F[1]} becomes
 \begin{align}
 \nabla^{\bas}_a(u) &= \nabla^F_u a + [a, u] = \widetilde{\nabla}^F_{\pr_F(u)}a + \pr_F [\pr_B(u),a] + [a, u] \notag \\
 &=  \widetilde{\nabla}^F_a \pr_F(u) + \pr_B[a, \pr_B(u)] = \widetilde{\nabla}^F_a \pr_F(u) + \nabla^{\Bott}_a \pr_B(u), \label{Eq: basic connection on tm}
 \end{align}
 and
 \begin{equation}\label{Eq: basic connection on F[1]}
 \nabla^{\bas}_a(a^\prime [1])= (\widetilde{\nabla}^F_a a^\prime)[1] = (\nabla^{\bas}_a a^\prime)[1] ,
 \end{equation}
respectively, for any $a, a^\prime \in \Gamma(F)$ and $u \in \Gamma(T_\k M)$.
Using Equation~\eqref{Eq: basic connection on tm} and the fact that $\widetilde{\nabla}^F$ is torsion-free, one has
\begin{align*}
  {R^{\bas}_{\nabla^F}}(a,a^\prime)(a^{\prime\prime}) &= \widetilde{\nabla}^F_{[a,a^\prime]} a^{\prime\prime} - \widetilde{\nabla}_a^F \widetilde{\nabla}_{a^\prime}^F a^{\prime\prime} +  \widetilde{\nabla}^F_{a^\prime} \widetilde{\nabla}^F_a a^{\prime\prime}= -R_{\widetilde{\nabla}^F}(a,a^\prime)a^{\prime\prime}.
\end{align*}
Meanwhile, by Equations~\eqref{Eq: nablaF} and~\eqref{Eq: basic connection on tm}, one also has
\begin{align*}
  & \qquad {R^{\bas}_{\nabla^F}}(a,a^\prime)(j(b)) \\
  &= \pr_F[j(b), [a,a^\prime]] - [\pr_F[j(b),a], a^\prime] - [a, \pr_F [j(b),a^\prime]] - \pr_F[\pr_B[a^\prime, j(b)], a] + \pr_F[\pr_B[a,j(b)], a^\prime] \notag \\
       &= \pr_F\left([j(b), [a, a^\prime]] - [[j(b), a], a^{\prime}] - [a, [j(b), a^\prime]]\right)= 0,
\end{align*}
for any $a, a^\prime \in \Gamma(F)$ and $b \in \Gamma(B)$.
This completes the proof.
\end{proof}
As a consequence, we have
\begin{corollary}
Given a splitting $j\colon B \to \TkM$ of the short exact sequence~\eqref{SES} and
a torsion-free $F$-connection$\widetilde{\nabla}^F$ on $F$, we have an isomorphism of cochain complexes
\begin{equation}\label{Eq: STF[1] as representation UTH}
  	\Big(\Gamma\left(S^kT_{F[1]}\right), \LdF\Big) \xrightarrow{\cong} \Big(\Omega_F\big(S^k(F[1] \oplus T_\k M)\big), D_{\nabla^F} = \delta + d_{\nabla^{\bas}} - R_{ \widetilde{\nabla}^F}\Big).
\end{equation}
\end{corollary}

\emph{Step 2 -- A basic contraction}.
Let $A$ be a commutative $\k$-algebra.  Assume that $U$ is an $A$-module and $V \subset U$ is an $A$ submodule such that the quotient $A$-module $U/V$ is projective. Then we have a split short exact sequence of $A$-modules
\begin{equation}\label{SES of VS}
0 \rightarrow V \xrightarrow{i} U \xrightarrow{\pr_{U/V}} U/V \rightarrow 0,
\end{equation}
and a $2$-term cochain complex of $A$-modules $ (V[1] \to U, \delta)$ concentrated in degrees $(-1)$ and $0$. Here $V[1] = \{v[1] \mid v \in V\}$ is the $A$-module obtained from $V$ by a degree shifting and the differential $\delta$ is simply the inclusion $v[1] \mapsto i(v)$  for any $v[1] \in V[1]$.

It is well-known that the $2$-term complex $(V[1] \to U, \delta)$ homotopy contracts to $U/V$. By taking symmetric tensor product, its $k$-th symmetric tensor product $S_A^k(V[1] \oplus U)$ homotopy contracts onto the $k$-th symmetric tensor product $S_A^k(U/V)$ of the $A$-module $U/V$ (cf. \cites{SX, Luca}).
For completeness, we sketch a proof below.
\begin{lemma}\label{Lem: contraction of vector spaces}
Any splitting $j \colon U/V \to U$ of the short exact sequence~\eqref{SES of VS} of $A$-modules induces a contraction for any $k \geq 1$:
	\begin{equation}\label{Eq: contraction of vector spaces}
	\begin{tikzcd}
	\big(S_A^k(V[1] \oplus U),\delta\big) \arrow[loop left, distance=2em, start anchor={[yshift=-1ex]west}, end anchor={[yshift=1ex]west}]{}{h_k} \arrow[r,yshift = 0.7ex, "\phi_k"] & \big(S_A^k(U/V),0\big) \arrow[l,yshift = -0.7ex, "\psi_k"].
	\end{tikzcd}
	\end{equation}
	Here $h_k$ and $\psi_k$ depend on the choice of $j$ while $\phi_k$ does not.
\end{lemma}
\begin{proof}
	Via the splitting $j \colon U/V \to U$, we have an isomorphism $U \cong V \oplus U/V$ of $A$-modules. Denote by $\pr_V \colon U \to V$ the associated projection onto $V$.
	Define three $A$-linear maps as follows:
	\begin{align*}
	\phi_1 \colon V[1] \oplus U &\to U/V, & \phi_1((v[1],u)) &= \pr_{U/V}(u),  \\
    \psi_1 \colon U/V &\to V[1] \oplus U, & \psi_1(\overline{u}) &= (0, j(\overline{u})), \\
    h_1 \colon V[1] \oplus U &\to (V[1] \oplus U)[1],  & h_1((v[1],u)) &= (-\pr_V(u)[1], 0).
	\end{align*}
	It is easy to see that the tripe $(\phi_1,\psi_1,h_1)$ defines a contraction:
	\[
	\begin{tikzcd}
	\big(V[1] \oplus U,\delta\big) \arrow[loop left, distance=2em, start anchor={[yshift=-1ex]west}, end anchor={[yshift=1ex]west}]{}{h_1} \arrow[r,yshift = 0.7ex, "\phi_1"] & (U/V,0) \arrow[l,yshift = -0.7ex, "\psi_1"].
	\end{tikzcd}
	\]
Applying the tensor trick~\cites{Manetti, Berglund} to the above contraction $(\phi_1,\psi_1,h_1)$, we obtain the desired contraction~\eqref{Eq: contraction of vector spaces}, where
	\begin{align*}
	\phi_k((v_1[1],u_1) \odot \cdots \odot (v_k[1],u_k)) &= \phi_1(v_1[1],u_1) \odot \cdots \odot \phi_1(v_k[1],u_k), \\
	\psi_k(\overline{u}_1 \odot \cdots \odot \overline{u}_k) &= \psi_1(\overline{u}_1) \odot \cdots \odot \psi_1(\overline{u}_k),
	\end{align*}
	for any $(v_i[1],u_i) \in V[1] \oplus U, \overline{u}_i \in U/V, 1 \leq i \leq k$, and
	\begin{align*}
	h_k (v_1[1] \odot \cdots \odot v_p[1] & \otimes v_{p+1} \odot \cdots \odot v_{p+q} \otimes \overline{u}_{p+q+1} \odot \cdots \odot \overline{u}_{k}) \\
	=  \frac{(-1)^p}{p+q}\sum_{i=1}^{q} v_1[1] \odot \cdots \odot v_p[1] \odot  h_1(v_{p+i}) & \otimes v_{p+1} \odot \cdots \odot \widehat{v_{p+i}} \odot \cdots \odot v_{p+q} \otimes \overline{u}_{p+q+1} \odot \cdots \odot \overline{u}_{k}  ,
	\end{align*}
	for any $p,q \geq 0, 0 < p+q \leq k$ and any $v_1[1],\cdots v_p[1] \in V[1],  v_{p+1}, \cdots, v_{p+q} \in V, \overline{u}_{p+q+1}, \cdots, \overline{u}_{k} \in U/V$.
\end{proof}

\emph{Step 3 -- The desired contraction}. We are now ready to complete the
\begin{proof}[Proof of Proposition~\ref{prop: coalgebra contraction on STF}]
It suffices to show that for any $k \geq 1$, there is a contraction
\begin{equation}\label{Eq: contraction of symmetric k-vector fields}
	\begin{tikzcd}
	\Big(\Gamma\left(S^kT_{F[1]}\right), \LdF\Big) \arrow[loop left, distance=2em, start anchor={[yshift=-1ex]west}, end anchor={[yshift=1ex]west}]{}{H_k} \arrow[r,yshift = 0.7ex, "\Phi_k"] & \Big(\OmegaF\left(S^k B\right) , d^{S^kB}_F\Big) \arrow[l,yshift = -0.7ex, "\Psi_k"],
	\end{tikzcd}
	\end{equation}
where $d_F^{S^kB}$ is the Chevalley-Eilenberg differential of the $F$-module $S^k B$.

Applying  Lemma~\ref{Lem: contraction of vector spaces} to the  $R$-module $U:= \Gamma({\TkM})$ and its   submodule $V:= \Gamma(F)$,   we  obtain a contraction of $R$-modules
  \[
   	\begin{tikzcd}
	\Big(\Gamma\big(S^k(F[1] \oplus {\TkM})\big), \delta\Big) \arrow[loop left, distance=2em, start anchor={[yshift=-1ex]west}, end anchor={[yshift=1ex]west}]{}{h_k} \arrow[r,yshift = 0.7ex, "\phi_k"] & \big(\Gamma\left(S^k B\right) , 0\big). \arrow[l,yshift = -0.7ex, "\psi_k"]
	\end{tikzcd}
  \]

The $\Omega_F$-linear extension of this contraction gives rise to a contraction of $\Omega_F$-modules
\begin{equation}\label{Eq: contraction of delta}
\begin{tikzcd}
	\Big(\Omega_F\big(S^k(F[1] \oplus {\TkM})\big), \delta\Big) \arrow[loop left, distance=2em, start anchor={[yshift=-1ex]west}, end anchor={[yshift=1ex]west}]{}{h_k} \arrow[r,yshift = 0.7ex, "\phi_k"] & \Big(\Omega_F(S^k B) , 0\Big) \arrow[l,yshift = -0.7ex, "\psi_k"].
\end{tikzcd}
\end{equation}
Here the differential $\delta$ is given by~\eqref{Eq: Def of delta}.

Observe that $d_{\nabla^{\bas}} - R_{\widetilde{\nabla}^F}$ is a perturbation of $\delta$.
By the definition of the operators ${R_{\widetilde{\nabla}^F}}$, $h_k$, and the basic $F$-connection on $F[1]$~\eqref{Eq: basic connection on F[1]}, it is easy to see that
\[
   h_k \left(d_{\nabla^{\bas}} - R_{\widetilde{\nabla}^F}\right) = - \left(d_{\nabla^{\bas}} - R_{\widetilde{\nabla}^F}\right) h_k \colon \Omega_F\left(S^k(F[1] \oplus {\TkM})\right) \to \Omega_F\left(S^k(F[1] \oplus {\TkM})\right).
\]
Combining with the side conditions $\phi_k  h_k=0, h_k \psi_k = 0$ and $h_k^2 = 0$, we have
\begin{align*}
 \phi_k  \left(d_{\nabla^{\bas}} - R_{\widetilde{\nabla}^F}\right) h_k &= - \phi_k h_k \left(d_{\nabla^{\bas}} - R_{\widetilde{\nabla}^F}\right) = 0, \\
 h_k  \left(d_{\nabla^{\bas}} - R_{\widetilde{\nabla}^F}\right)  \psi_k &= - \left(d_{\nabla^{\bas}} - R_{\widetilde{\nabla}^F}\right)  h_k  \psi_k = 0, \\
 h_k \left(d_{\nabla^{\bas}} - R_{\widetilde{\nabla}^F}\right) h_k &= -h_k^2 \left(d_{\nabla^{\bas}} - R_{\widetilde{\nabla}^F}\right) = 0.
\end{align*}
Thus, the maps $\phi_k, \psi_k, h_k$, and the perturbation $d_{\nabla^{\bas}} - R_{\widetilde{\nabla}^F}$ satisfy the constraints in~\eqref{Eq: perturb constraints}.
Applying the perturbation Lemma~\ref{Lem: OPT} to the contraction~\eqref{Eq: contraction of delta} and the perturbation $d_{\nabla^{\bas}} - R_{\widetilde{\nabla}^F}$, we obtain a new contraction
\[
\begin{tikzcd}
	\Big(\Omega_F\big(S^k(F[1] \oplus {\TkM})\big), D_{\nabla^F} = \delta+ d_{\nabla^{\bas}} - R_{\widetilde{\nabla}^F}\Big) \arrow[loop left, distance=2em, start anchor={[yshift=-1ex]west}, end anchor={[yshift=1ex]west}]{}{h^\prime_k} \arrow[r,yshift = 0.7ex, "\phi^\prime_k"] & \Big(\Omega_F(S^k B) , d_B^\prime\Big) \arrow[l,yshift = -0.7ex, "\psi^\prime_k"],
\end{tikzcd}
\]
where
\begin{align*}
  \phi^\prime_k &= \sum_{l =0}^{\infty}\phi_k \left(\left(d_{\nabla^{\bas}} - R_{\widetilde{\nabla}^F}\right) h_k\right)^l = \phi_k + \sum_{l=1}^\infty \big(\phi_k   \left(d_{\nabla^{\bas}} - R_{\widetilde{\nabla}^F}\right) h_k\big)  \big(\left(d_{\nabla^{\bas}} - R_{\widetilde{\nabla}^F}\right) h_k\big)^{l-1} = \phi_k, \\
  \psi^\prime_k &= \sum_{l=0}^{\infty}\left(h_k \left(d_{\nabla^{\bas}} - R_{\widetilde{\nabla}^F}\right)\right)^l \psi_k = \psi_k + \sum_{l=1}^{\infty} \left(h_k \left(d_{\nabla^{\bas}} - R_{\widetilde{\nabla}^F}\right)\right)^{l-1} \big(h_k  \left(d_{\nabla^{\bas}} - R_{\widetilde{\nabla}^F}\right)  \psi_k\big) = \psi_k, \\
  h_k^\prime &= \sum_{l=0}^{\infty} \left(h_k  \left(d_{\nabla^{\bas}} - R_{\widetilde{\nabla}^F}\right)\right)^l  h_k = h_k + \sum_{l=1}^{\infty} \left(h_k \left(d_{\nabla^{\bas}} - R_{\widetilde{\nabla}^F}\right)\right)^{l-1} \big(h_k  \left(d_{\nabla^{\bas}} - R_{\widetilde{\nabla}^F}\right)  h_k\big) = h_k,
  \end{align*}
  and the new differential on $\Omega_F(S^k B)$ coincides with the Chevalley-Eilenberg differential of the $F$-module $S^k B$:
  \begin{align*}
  d_B^\prime &= \phi^\prime_k \left(d_{\nabla^{\bas}} - R_{\widetilde{\nabla}^F}\right)  \psi_k = \phi_k  d_{\nabla^{\bas}} \psi_k = d_F^{S^k B}.
\end{align*}
Here we have used the fact that the basic $F$-connection $\nabla^{\bas}$ on $T_\k M$ defined by Equation~\eqref{Eq: basic connection on tm} extends the Bott $F$-connection $\nabla^{\Bott}$.
Hence, we obtain a contraction of $\Omega_F$-modules
\[
 \begin{tikzcd}
	\Big(\Omega_F\big(S^k(F[1] \oplus {\TkM})\big), D_{\nabla^F}\Big) \arrow[loop left, distance=2em, start anchor={[yshift=-1ex]west}, end anchor={[yshift=1ex]west}]{}{h_k} \arrow[r,yshift = 0.7ex, "\phi_k"] & \Big(\Omega_F(S^k B) , d_F^{S^k B}\Big) \arrow[l,yshift = -0.7ex, "\psi_k"].
\end{tikzcd}
\]
Combining with the isomorphism in~\eqref{Eq: STF[1] as representation UTH}, we obtain the desired contraction~\eqref{Eq: contraction of symmetric k-vector fields}.
\end{proof}
\begin{remark}
In fact, the contraction~\eqref{Eq: contraction in CXX} does not depend on the choice of the $F$-connection $\widetilde{\nabla}^F$ on $F$ (see~\cites{CXX,Luca} for the explicit construction without choosing such an $F$-connection).
Let us consider the $k=1$ case, that is, the contraction of vector fields on $F[1]$.
Under the identification~\eqref{Eqt:splitXFone}, the left $\OmegaF$-module $\Gamma(T_{F[1]})$ is generated by two types of derivations on $\Omega_F$:
\begin{equation}\label{Eq: generators of TF[1]}
\{\widehat{u} \mid  u \in\sections{T_\k M}\}, \quad\mbox{and} \quad \{\iota_a \mid a \in \sections{F} \},
\end{equation}
which are of homogeneous degrees $0$ and $(-1)$, respectively.
Here $\widehat{u} \in \Gamma(T_{F[1]}) \cong \Der(\Omega_F)$ maps $R$ to $R$ and linear functions $\xi \in \Gamma(F^\vee)$ on $F[1]$ to linear functions. More precisely, for any $f \in R, \xi \in \Gamma(F^\vee)$, and $a \in \Gamma(F)$, we have
\begin{align}\label{Eq: Def of horizontal lifting}
\widehat{u}(f) &= u(f), & \langle \widehat{u}(\xi), a\rangle = u \langle \xi, a \rangle - \langle u, \nabla_u^F a \rangle.
\end{align}
Restricting to the case $k=1$,   the contraction in~\eqref{Eq: contraction in CXX} is determined by the following simple formulas:
	\begin{align}
	\Phi(\iota_a) &= 0,	& \Phi(\widehat{u}) &= \pr_B(u),   \notag \\
    \Psi(b) &= \widehat{j(b)},  & & \label{Eq: Psifork=1} \\
	H(\iota_a) &= 0,      &  H(\widehat{u}) &= -\iota_{\pr_F(u)}. \notag
	\end{align}	
\end{remark}

\begin{remark}
Given a splitting $j$ of the short exact sequence~\eqref{SES} and a $T_\k M$-connection $\nabla^F$ on $F$, one has an isomorphism
\begin{equation}\label{Eqt:SXXFoneijk}
\Gamma\left(ST_{F[1]}\right) \cong \bigoplus_{i,j,k\geq 0}\OmegaF\left(S^i F \otimes  S^j B \otimes S^kF[1]\right).
\end{equation}
It is simple to see that $H$ satisfies the following condition:
\begin{equation}\label{Eqt:Hijkto}
	H\left(\OmegaF\left(S^i F \otimes S^j B  \otimes S^kF[1] \right)\right) \subseteq
	\OmegaF\left(S^{i-1} F \otimes S^j B  \otimes S^{k+1} F[1]\right).
\end{equation}
\end{remark}

\subsubsection{Two PBW isomorphisms}
We now recall the Poincar\'{e}-Birkhoff-Witt isomorphism for the graded manifold $F[1]$~\cite{LS} and that for the Lie pair $({\TkM}, F)$~\cite{LGSX}. The construction of both PBW isomorphisms needs \emph{a priori} certain connections.

We first introduce a special affine connection on the graded manifold ${F[1]}$.
A triple $\big(j, \widetilde{\nabla}^F, \nabla^B\big)$ consists of the following data:
\begin{compactenum}
	\item  a splitting $j\colon B \rightarrow {\TkM}$ of the short exact sequence~\eqref{SES};
	\item  a torsion-free $F$-connection $\widetilde{\nabla}^F$ on $F$;
	\item  a linear connection $\nabla^B$ on $B$ extending the Bott $F$-connection.
\end{compactenum}
\begin{lemma}\label{lem: connection on E}
	Any triple $\big(j,\widetilde{\nabla}^F, \nabla^B\big)$ induces a linear connection $\nabla^E$ on the vector bundle $E:= {\TkM} \oplus F[1]$ over $M$.
\end{lemma}
\begin{proof}
	Recall that the pair $\big(j, \widetilde{\nabla}^F\big)$ determines a ${\TkM}$-connection $\nabla^F$ on $F$ defined as in~\eqref{Eq: nablaF}.
	Together with the linear connection $\nabla^B$ on $B$, we obtain a ${\TkM}$-connection \footnote{This linear connection on ${\TkM}$ is called an adapted connection in~\cite{Luca}.} $\nabla^{{\TkM}}$ on ${\TkM}$ defined by
	\[
	\nabla^{{\TkM}}_u v = \nabla^F_{u} \pr_F(v) + \nabla^B_u \pr_B(v),
	\]
	for any $u, v \in \Gamma({\TkM})$.
	Consider the graded vector bundle $E := {\TkM} \oplus  F[1]$ over $M$. The ${\TkM}$-connection $\nabla^F$ on $F$ induces a linear connection $\nabla^{F[1]}$ on the graded vector bundle $F[1]$ over $M$.
	Then $\nabla^{E} :=  \nabla^{{\TkM}}+ \nabla^{F[1]}$ defines a linear connection on $E$.
\end{proof}

\begin{proposition}\label{prop: pullback connection}
  Any triple $\big(j, \widetilde{\nabla}^F, \nabla^B\big)$ induces an affine connection $\nabla$ on the graded manifold $F[1]$.
\end{proposition}
\begin{proof}
Recall that given the pair $\big(j,\widetilde{\nabla}^F\big)$, the linear connection $\nabla^F$ on $F$ defined by Equation~\eqref{Eq: nablaF} induces an isomorphism of graded vector bundles over the graded manifold $F[1]$
\begin{align}\label{Eq: splitfor tangent bundle}
T_{F[1]} &\cong \pi^\ast({\TkM}\oplus F[1]) = \pi^\ast(E).
\end{align}
The linear connection $\nabla^E$ on $E$ in Lemma~\ref{lem: connection on E} induces a pullback $T_{F[1]}$-connection $\pi^\ast(\nabla^E)$ on the pullback bundle $\pi^\ast(E)$ over the graded manifold $F[1]$. By the isomorphism~\eqref{Eq: splitfor tangent bundle}, this pullback connection $\pi^\ast(\nabla^E)$ determines an affine connection $\pullbackconn$ on the graded manifold $F[1]$.
\end{proof}
The affine connection in the above proposition is called the \emph{pullback connection} on the graded manifold $F[1]$ associated with the chosen triple $\big(j,\widetilde{\nabla}^F, \nabla^B\big)$.
Using the identification~\eqref{Eqt:splitXFone}, we obtain an explicit expression of the affine connection
\[
\pullbackconn \colon  \Gamma\left(T_{F[1]}\right)\times \Gamma\left(T_{F[1]}\right) \to \Gamma\left(T_{F[1]}\right),
\]
in terms of the generators of $\Gamma\left(T_{F[1]}\right)$ as in~\eqref{Eq: generators of TF[1]}:
\begin{equation}\label{Eqt:tableofpullbackconn}
\left\{
\begin{array}{ll}
\pullbackconn_{\iota_a}(\iota_{a^\prime})=0, &   \\
&\\
\pullbackconn_{\iota_a}(\widehat{u})=0, &    \\
&\\
\pullbackconn_{\widehat{u}}(\iota_a)=\iota_{\nabla^F_u a}, &   \\
&\\
\pullbackconn_{\widehat{u}}(\widehat{u^\prime})=\widehat{\nabla^{\TkM}_{u}{u^\prime}}, &
\end{array}
\right.
\end{equation}
for any $a, a^\prime \in \sections{F}$ and any $u, u^\prime \in \Gamma({\TkM})$.
In particular, we have, for any $b, b^\prime \in \Gamma(B)$,
\begin{equation}\label{Eq: Bconnection and pullback connection}
	\pullbackconn_{\widehat{j(b)}} \widehat{j(b^\prime)} = \widehat{\nabla^B_{j(b)} b^\prime} = \Psi (\nabla^B_{j(b)} b^\prime).
\end{equation}

We will fix a triple $\big(j,\widetilde{\nabla}^F, \nabla^B\big)$ in the sequel.
Equipping the graded manifold $F[1]$ with the connection $\pullbackconn$ as in Proposition~\ref{prop: pullback connection}, we obtain an isomorphism of filtered ${\OmegaF}$-coalgebras
\begin{equation}\label{Eq: pbw}
\operatorname{pbw} \colon  \Gamma\left(ST_{F[1]}\right) \rightarrow \D({F[1]}),
\end{equation}
called the PBW isomorphism for the graded manifold $F[1]$, and defined by the inductive recipe (see~\cite{LS}):
\begin{align*}
\pbw(\omega) &= \omega,  \quad \forall \omega \in \OmegaF \cong C^{\infty}({F[1]}), \\
\pbw(X) &= X, \quad \forall X \in \Gamma(T_ {F[1]}),
\end{align*}
\begin{equation}\label{Eqt:pbw0ton}\mbox{and}\quad
\pbw(X_1 \odot  \cdots \odot X_n)  = \frac{1}{n}\sum_{k=1}^n\epsilon_k \left\{X_k \circ \pbw(X^{\{k\}}) - \pbw\big(\pullbackconn_{X_k}(X^{\{k\}})\big) \right\},
\end{equation}
where $\epsilon_k = (-1)^{\abs{X_k}(\abs{X_1}+\cdots+\abs{X_{k-1}})}$ and $X^{\{k\}} = X_1 \odot \cdots \odot X_{k-1}\odot X_{k+1}\odot\cdots \odot X_n$.

Meanwhile, according to~\cite{LGSX}, the pair $\big(j,\nabla^B\big)$ in the chosen triple determines an isomorphism of $R$-coalgebras
\[
\overline{\pbw}\colon \Gamma(SB) \rightarrow \frac{\D(M)}{\D(M)\Gamma(F)} = \D(B),
\]
called the PBW isomorphism for the Lie pair $({\TkM}, F)$ therein, which can be defined inductively as follows:
\begin{align*}
\overline{\pbw}(f) &= f,  \quad \forall f \in R, \\
\overline{\pbw}(b) &= j(b) , \quad \forall b \in \Gamma(B),    \\
\mbox{and }\quad \overline{\pbw}(b_1 \odot  \cdots \odot  b_n) &= \frac{1}{n}\sum_{k=1}^n \left\{ j(b_k)  \circ \overline{\pbw}(b^{\{k\}}) - \overline{\pbw}\big(\nabla^B_{ j(b_k) }(b^{\{k\}})\big) \right\},
\end{align*}
where $b^{\{k\}} = b_1 \odot  \cdots \odot  b_{k-1} \odot  b_{k+1} \odot  \cdots \odot  b_n$. Via an $\OmegaF$-linear extension, we obtain an isomorphism of ${\OmegaF} $-coalgebras, denoted by the same symbol by abuse of notation
\begin{equation}\label{Eq: overlinepbw}
\overline{\pbw}\colon \OmegaF(SB) \rightarrow \OmegaF(\D(B)).
\end{equation}
Now we have two PBW isomorphisms $\pbw$~\eqref{Eq: pbw} and $\overline{\pbw}$~\eqref{Eq: overlinepbw}. Next we investigate how they are related.
To do so, we introduce a map
\begin{align}\label{Eq: phinatural}
\Phi_\natural \colon \D({F[1]})  &\to \OmegaF(\D(B)) = \OmegaF \otimes_R\frac{\D(M)}{\D(M)\Gamma(F) },  \\
D  &\mapsto \overline{\pi_\ast (D)}, \notag
\end{align}
where $\pi_\ast\colon \D({F[1]}) \to \OmegaF \otimes_R \D(M)$ is the restriction map determined by
\begin{equation}\label{Equ: restriction pi}
\pi_\ast(D)(f) = D(\pi^\ast f),
\end{equation}
for any $f \in R$, and $\overline{\pi_\ast (D)} \in \Omega_F(\D(B))$ denotes the class of $\pi_\ast (D)$ in $\OmegaF \otimes_R \frac{\D(M)}{\D(M)\Gamma(F)}$.

\begin{lemma}\label{lem:Phinaturalchainmap}
The map $\Phi_\natural$ is a morphism of filtered dg coalgebras over the commutative dg algebra $(\Omega_F, d_F)$ from $(\D({F[1]}), \llbracket d_F,- \rrbracket)$ to $(\OmegaF(\D(B)), \dF^\U)$.
\end{lemma}
\begin{proof}
 In fact, for any $D \in \D(F[1])$ and $f, g \in R$, we have
 \begin{align*}
   \pi_\ast^{\otimes 2}(\Delta(D))(f \otimes g) &= \Delta(D)(\pi^\ast f \otimes \pi^\ast g) = D(\pi^\ast(fg)),
 \end{align*}
 and
 \begin{align*}
   \Delta(\pi_\ast(D))(f \otimes g) &= \pi_\ast(D)(fg) = D(\pi^\ast(fg)).
 \end{align*}
 It follows that $\pi_\ast^{\otimes 2} \circ \Delta = \Delta \circ \pi_\ast$, i.e., $\pi_\ast  \colon \D(F[1]) \to \Omega_F(\D(M))$ is a morphism of $\Omega_F$-coalgebras.
Meanwhile,  the projection $\Omega_F(\D(M)) \to \Omega_F(\D(B))$ is a morphism of $\Omega_F$-coalgebras by definition.  It thus follows that $\Phi_\natural$ is a morphism of $\Omega_F$-coalgebras as well.

We now show that $\Phi_\natural$ intertwines the two differentials:
\begin{equation*}
\Phi_\natural \left(\llbracket \dF,D \rrbracket\right) = \dF^\U \left(\Phi_\natural(D)\right),\qquad \forall D \in \D({F[1]}).
\end{equation*}
Without loss of generality, we may assume that $\pi_\ast (D)=\omega\otimes D_0$ for some homogeneous $\omega\in \OmegaF$ and $D_0\in \D(M)$. Then we have, for any $f \in R$,
	\[
	  D(\pi^\ast f) = \pi_\ast(D) (f) = D_0(f) \omega,
	\]
and therefore
	\begin{eqnarray*}
		\pi_\ast(\dF \circ D) (f) &=&  (\dF \circ D)(\pi^\ast f)=\dF(D_0(f)\omega)\\
                                                 &=& \dF(D_0 f) \wedge \omega + (D_0 f) \dF\omega,
	\end{eqnarray*}
which implies that
	\[
	\pi_\ast(\dF \circ D)= \dF\omega \otimes D_0 + (-1)^{\abs{\omega}}\omega \wedge (\pi_\ast(\dF) \circ D_0).
	\]
Meanwhile, it follows from Equation~\eqref{Eq:piastdF} that
\[
  \pi_\ast (D \circ d_F) \subseteq \Omega_F\otimes_R (\D(M)\Gamma(F)).
\]
Therefore, we have
	\begin{eqnarray*}
		\Phi_\natural \left(\llbracket \dF, D \rrbracket\right)  &=& \overline{\pi_\ast(\dF \circ D) - (-1)^{\abs{D}}\pi_\ast(D \circ \dF)} \\
	                                           &=& \overline{\pi_\ast(\dF \circ D)} \\
                                               &=& \dF\omega \otimes \overline{D_0}+ (-1)^{\abs{\omega}}\omega \wedge \overline{\pi_\ast(\dF) \circ D_0} \qquad (\text{by Equation~\eqref{Eq: dFU}})\\
                                               &=& {\dF^\U}( \omega\otimes \overline{D_0}) =  {\dF^\U}\left(\Phi_\natural(D)\right).
	\end{eqnarray*}
This completes the proof.
\end{proof}

We also need some properties of the restriction map $\pi_\ast \colon \D(F[1]) \to \Omega_F \otimes_R \D(M)$. For simplicity, we introduce the notation:
\begin{equation}\label{Eq: ijk}
[i, j, k] := \Gamma (S^i F \otimes S^j B  \otimes S^k F[1]),
\end{equation}
for any $i, j, k \geq 0$. Then the $\Omega_F$-module $\Gamma(ST_{F[1]})$ is generated by these $[i,j,k]$ under the identification~\eqref{Eqt:SXXFoneijk}.
\begin{lemma}
  For any $u \in \Gamma(T_\k M), v[1] \in \Gamma(F[1])$ and any $i,j,k \geq 0$,
\begin{align}\label{Eq: piast preserves product}
  &\pi_\ast(\widehat{u} \circ \pbw[i,j,k]) = u \circ \pi_\ast(\pbw[i,j,k]), \\
  &\pi_\ast(I(v[1]) \circ \pbw[i,j,k]) = 0.  \label{Eq: piastonF[1]}
\end{align}
Here $\widehat{u} \in \mathfrak{X}(F[1])$ is the horizontal lifting of $u$ as defined in Equation~\eqref{Eq: Def of horizontal lifting}, and $I \colon \Gamma(F[1]) \hookrightarrow \Gamma(T_{F[1]})$ is the canonical vertical lifting as in~\eqref{Eq: Def of I}.
\end{lemma}
\begin{proof}
We first prove that $\pbw[i,j,k] \subseteq \D(F[1])$ are all projectable differential operators, that is,
\[
\pi_\ast(\pbw[i,j,k]) \subseteq \D(M).
\]
Note that
\[
 \pbw[1,0,0] = \widehat{\Gamma(F)}, \quad \pbw[0,1,0] = \widehat{\Gamma(B)}, \quad \pbw[0,0,1] = I(\Gamma(F[1])).
\]
Thus, we have
\[
 \pi_\ast(\pbw[1,0,0]) = \Gamma(F) \subseteq \D(M), \quad \pi_\ast(\pbw[0,1,0]) = \Gamma(B) \subseteq \D(M), \quad \pi_\ast(\pbw[0,0,1]) = 0.
\]
By Equations~\eqref{Eqt:tableofpullbackconn} and~\eqref{Eqt:pbw0ton}, one has
\begin{align}\label{Eq: pbwijk}
    \pbw [i,j,k] &\subseteq \widehat{\Gamma(F)}\circ \pbw[i-1,j,k] + \widehat{\Gamma(B)}\circ \pbw[i,j-1,k] + I(\Gamma(F[1])) \circ \pbw[i,j,k-1] \notag \\
	&\qquad +\pbw[i-1,j,k] + \pbw [i,j-1,k].
\end{align}
Since projectable differential operators on $F[1]$ are closed under composition, it follows by induction that $\pbw[i,j,k] \subseteq \D(F[1])$ are projectable for any $i,j,k \geq 0$.
Hence, we have
\begin{align*}
  \pi_\ast(\widehat{u} \circ \pbw[i,j,k]) &= \pi_\ast(\widehat{u}) \circ \pi_\ast(\pbw[i,j,k]) = u \circ \pi_\ast(\pbw[i,j,k]), \quad \mbox{and} \\
  \pi_\ast(I(v[1]) \circ \pbw[i,j,k]) &= \pi_\ast(I(v[1])) \circ \pi_\ast(\pbw[i,j,k]) = 0.
\end{align*}
\end{proof}

The first relation between these two PBW isomorphisms is given by the following
\begin{lemma}[\cite{Luca}*{Remark 32}]\label{Eq: Phinatural and Psi}
The PBW isomorphisms $\pbw$~\eqref{Eq: pbw} and $\overline{\pbw}$~\eqref{Eq: overlinepbw} are related as follows:
\[
	\overline{\pbw} = \Phi_\natural \circ {\pbw}	\circ \Psi \colon \Omega_F(SB) \to \Omega_F(\D(B)),
\]
where $\Phi_\natural$ is defined by~\eqref{Eq: phinatural} and $\Psi$ is the inclusion in the contraction~\eqref{Eq: contraction in CXX}.
\end{lemma}
\begin{proof}
Since all maps involved are $\OmegaF$-linear, it suffices to prove the identity
\begin{equation}\label{temp1}
	(\Phi_\natural\circ {\pbw}	\circ\Psi) (b_1 \odot \cdots \odot b_n) = \overline{\pbw}(b_1 \odot  \cdots   \odot  b_n),
\end{equation}
for any  $b_1, \cdots, b_n \in \Gamma(B)$.
We argue by induction. The $n=1$ case is obvious as both sides of Equation~\eqref{temp1} yields $b_1$. Now assume that~\eqref{temp1} holds for some $n \geq 1$.
Then we proceed to the $n+1$ case.
Note that
\[
(\pbw \circ \Psi)\left(b_1 \odot \cdots \odot b_n\right) = \pbw\left(\widehat{j(b_1)} \odot \cdots \odot \widehat{j(b_n)}\right) \in \pbw[0,n,0] \subseteq \D(F[1]).
\]
Thus, we have
\begin{align*}
	&\quad\; (\Phi_\natural\circ \pbw \circ\Psi) \left(b_1 \odot \cdots \odot b_{n+1}\right) \\
   &= (\Phi_\natural\circ \pbw)\left(\widehat{j(b_1)} \odot \cdots \odot \widehat{j(b_{n+1})}\right) \qquad\qquad\qquad \left(\text{by Equation~\eqref{Eqt:pbw0ton}}\right) \\
	  &= \frac{1}{n+1}\sum_{k=1}^{n+1} \Phi_\natural \left\{ \widehat{j(b_k)} \circ \pbw (\widehat{j(b)^{\{k\}}}) - \pbw(\pullbackconn_{\widehat{j(b_k)}} \widehat{j(b)^{\{k\}}}) \right\} \\
     &= \frac{1}{n+1}\sum_{k=1}^{n+1} \overline{\pi_\ast \left(\widehat{j(b_{k})} \circ \pbw  (\widehat{j(b)^{\{k\}}})\right)} - \Phi_\natural \left(\pbw( \pullbackconn_{\widehat{j(b_k)}} \widehat{j(b)^{\{k\}}}) \right) \qquad\qquad\left(\text{by Equation~\eqref{Eq: piast preserves product}}\right)\\
      &= \frac{1}{n+1}\sum_{k=1}^{n+1} \overline{j(b_k) \circ \pi_\ast \left(\pbw (\widehat{j(b)^{\{k\}}})\right)} - \Phi_\natural \left(\pbw(\pullbackconn_{\widehat{j(b_k)}} \widehat{j(b)^{\{k\}}}) \right) \qquad\qquad \left(\text{by Equation~\eqref{Eq: Bconnection and pullback connection}}\right)\\
      &= \frac{1}{n+1}\sum_{k=1}^{n+1}  j(b_k) \circ (\Phi_\natural\circ \pbw \circ\Psi) \left(b^{\{k\}}\right) - (\Phi_\natural\circ \pbw \circ\Psi) \left(\nabla^B_{j(b_k)} b^{\{k\}}\right)  \left(\text{by the inductive assumption}\right)\\
      &= \frac{1}{n+1}\sum_{k=1}^{n+1} j(b_{k}) \circ \overline{\pbw}\left( b^{\{k\}}\right) - \overline{\pbw} \left(\nabla^B_{j(b_k)}  b^{\{k\}}\right) \\
      &= \overline{\pbw}\left(b_1 \odot \cdots \odot b_{n+1}\right).
\end{align*}
Here $\widehat{j(b)^{\{k\}}}$ denotes $\widehat{j(b_1)} \odot \cdots \odot \widehat{j(b_{k-1})} \odot \widehat{j(b_{k+1})} \odot \cdots \odot \widehat{j(b_{n+1})}$.
\end{proof}

\subsubsection{Concluding via homological perturbation}
Conjugating the contraction~\eqref{Eq: contraction in CXX} via the two filtered isomorphisms $\pbw$~\eqref{Eq: pbw} and $\overline{\pbw}$~\eqref{Eq: overlinepbw}, we obtain the filtered contraction of the $\Omega_F$-module $\D(F[1])$:
\begin{equation}\label{Eq: contraction on DF[1] before perturb}
  	\begin{tikzcd}
	\Big(\D(F[1]), d_0 := \pbw \circ L_{d_F} \circ \pbw^{-1}\Big) \arrow[loop left, distance=2em, start anchor={[yshift=-1ex]west}, end anchor={[yshift=1ex]west}]{}{H_0} \arrow[r,yshift = 0.7ex, "\Phi_0"] & \left(\OmegaF(\D(B)), \overline{d}_0 := \overline{\pbw} \circ d^{SB}_F \circ \overline{\pbw}^{-1}\right) \arrow[l,yshift = -0.7ex, "\Psi_0"],
	\end{tikzcd}
\end{equation}
where
\begin{align}\label{Eq: Def of phi0, psi0 and H0}
 \Phi_0 &= \overline{\pbw} \circ \Phi \circ \pbw^{-1}, & \Psi_0 &= \pbw \circ \Psi \circ \overline{\pbw}^{-1}, & H_0 &= \pbw \circ H \circ \pbw^{-1}.
\end{align}

\begin{lemma}
  The inclusion map $\Psi_0$ and the chain homotopy $H_0$ defined above are compatible with the projection $\Phi_\natural$~\eqref{Eq: phinatural} in the following sense:
\begin{align}\label{Eq: Phinatural and psi0}
  \Phi_\natural \circ \Psi_0 &= \id \colon \Omega_F(\D(B)) \to \Omega_F(\D(B)), \\ \label{Eq: phinatural and H0}
  \Phi_\natural  \circ H_0 &= 0 \colon \D(F[1]) \to \Omega_F(\D(B)).
\end{align}
\end{lemma}
\begin{proof}
  By Lemma~\ref{Eq: Phinatural and Psi}, we have
  \[
    \Phi_\natural \circ \Psi_0 = \Phi_\natural \circ \pbw \circ \Psi \circ \overline{\pbw}^{-1} = \overline{\pbw} \circ \overline{\pbw}^{-1} = \id.
  \]
  We now prove~\eqref{Eq: phinatural and H0}. By the definition of $\Phi_\natural$ and $H_0$, it suffices to show that
  \begin{equation}\label{Eq: piastpbwH=0}
   \pi_\ast  \circ \pbw \circ H = 0 \colon \Gamma(ST_{F[1]}) \to \Omega_F(\D(M)),
  \end{equation}
  where  $\pi_\ast$ is the restriction map defined as in~\eqref{Equ: restriction pi} and $H$ is the homotopy operator in~\eqref{Eq: contraction in CXX}.

 Recall that under the identification~\eqref{Eqt:SXXFoneijk}, one has
 \[
   H\left(\Gamma\left(ST_{F[1]}\right)\right) \subseteq  \bigoplus_{i\geq 1, j, k \geq 0}\OmegaF\left( S^{i-1} F \otimes S^j B  \otimes S^{k+1}F[1] \right)
 \]
 according to~\eqref{Eqt:Hijkto}. Thus, Equation~\eqref{Eq: piastpbwH=0} holds if we could show that
  \[
    \Big(\pi_\ast\circ \pbw \Big) \left(\Omega_F(S^i F \otimes S^j B  \otimes S^kF[1])\right) = 0,
  \]
  for any $i, j \geq 0$ and $k\geq 1$. Since both $\pbw$ and $\pi_\ast$ are $\Omega_F$-linear, it suffices to prove
\begin{equation}\label{Eqt:Phinaturalpbwijk2}
	\pi_\ast(\pbw([i,j,k]))=0,\qquad \forall i,j \geq 0, k \geq 1,
\end{equation}
where the $R$-module $[i,j,k]$ is as in~\eqref{Eq: ijk}. We proceed by induction on $n = i+j\geq 0$. The base case $n=0$, i.e., $(i,j,k) = (0,0,k)$, is obvious, because
\[
\pi_\ast (\pbw([0,0,k]) \subseteq \pi_\ast\big(\underbrace{I(\Gamma(F[1])) \circ \cdots \circ I(\Gamma(F[1]))}_{k  \mbox{ times }}\big) = 0,
\]
by Equation~\eqref{Eq: pbwijk}, where $I \colon \Gamma(F[1]) \hookrightarrow \Gamma(T_{F[1]})$ is the canonical vertical lifting defined in~\eqref{Eq: Def of I}.

Meanwhile, by~\eqref{Eq: pbwijk}, one has
\begin{align*}
  \pi_\ast(\pbw [i,j,k]) &\subseteq \pi_\ast\Big(\widehat{\Gamma(F)}\circ\pbw[i-1,j,k] + \widehat{\Gamma(B)}\circ \pbw[i,j-1,k] +I(\Gamma(F[1])) \circ \pbw[i,j,k-1]\Big) \\
  &\quad +\pi_\ast(\pbw[i-1,j,k]) + \pi_\ast(\pbw [i,j-1,k]) \qquad\qquad  \left(\text{by Equations~\eqref{Eq: piast preserves product} and~\eqref{Eq: piastonF[1]}}\right)\\
    &= \Gamma(F) \circ \pi_\ast(\pbw[i-1,j,k])  + \Gamma(B) \circ \pi_\ast(\pbw[i,j-1,k]) \notag \\
	&\quad +\pi_\ast(\pbw[i-1,j,k]) + \pi_\ast(\pbw [i,j-1,k]),
\end{align*}
It follows that Equation~\eqref{Eqt:Phinaturalpbwijk2}  holds for $i + j = n+1$ if it holds for $i+j = n$. This completes the proof.
\end{proof}

We are now ready to finish the proof of Theorem~\ref{thm: Luca's result}.
\begin{proof}[Proof of Theorem~\ref{thm: Luca's result}]
According to~\cites{MSX, SSX}, the perturbation $\Theta:= \llbracket d_F,- \rrbracket - d_0 \colon \D(F[1]) \to \D(F[1])$ of $d_0$ lowers the filtration in~\eqref{Eq: filtration on DF[1]} by $1$, that is, $\Theta(\D^{\leq k+1}(F[1])) \subseteq \D^{\leq k}(F[1])$.
Applying the homological perturbation Lemma~\ref{Lem:HP} to the filtered contraction~\eqref{Eq: contraction on DF[1] before perturb}, we obtain a new filtered contraction
\[
\begin{tikzcd}
\big(\D(F[1]), \llbracket d_F,- \rrbracket = d_0 + \Theta\big) \arrow[loop left, distance=2em, start anchor={[yshift=-1ex]west}, end anchor={[yshift=1ex]west}]{}{H_\flat} \arrow[r,yshift = 0.7ex, "\Phi_\flat"] & \big(\OmegaF(\D(B)) , \overline{d}_0 + \Theta_\flat\big) \arrow[l,yshift = -0.7ex, "\Psi_\flat"],
\end{tikzcd}
\]
where
\begin{align}
\Phi_\flat  &= \sum_{k\geq0}\Phi_0(\Theta H_0)^k,  &
\Theta_\flat  &= \sum_{k\geq 0}\Phi_0(\Theta H_0)^k \Theta \Psi_0 = \Phi_\flat\Theta\Psi_0, \notag \\
\Psi_\flat  &=\sum_{k\geq0}(H_0 \Theta)^k\Psi_0,  &
H_\flat &= \sum_{k\geq0}H_0(\Theta H_0)^k. \label{Eq: Def of Psi and Hnatural}
\end{align}
First of all, we prove that the perturbed projection $\Phi_\flat$ coincides with the projection $\Phi_\natural$. In fact, using~\eqref{Eq: Phinatural and psi0} and~\eqref{Eq: phinatural and   H0}, one has
\[
  \Phi_\natural  \Psi_\flat = \Phi_\natural  \sum_{k\geq0} (H_0 \Theta)^k\Psi_0 = \Phi_\natural \Psi_0 + \sum_{k\geq 1} \Phi_\natural H_0 \Theta (H_0 \Theta)^{k-1}\Psi_0 = \Phi_\natural \Psi_0 = \id,
\]
and
\[
  \Phi_\natural H_\flat = \Phi_\natural \sum_{k\geq 0}H_0 (\Theta H_0)^k = \sum_{k\geq 0}\Phi_\natural H_0 \Theta  (H_0 \Theta)^{k} = 0.
\]
Thus, we have
\begin{align*}
  \Phi_\natural - \Phi_\flat &= \Phi_\natural (\id - \Psi_\flat \Phi_\flat) = \Phi_\natural  (\llbracket d_F,- \rrbracket  H_\flat + H_\flat  \llbracket d_F, - \rrbracket ) \quad \left(\text{by Lemma~\ref{lem:Phinaturalchainmap}}\right)\\
  &= d_{F}^\U  \Phi_\natural  H_\flat + \Phi_\natural  H_\flat \llbracket d_F, - \rrbracket = 0.
\end{align*}
Finally, since according to Lemma~\ref{lem:Phinaturalchainmap}, $\Phi_\natural \colon (\D(F[1]), \llbracket d_F,- \rrbracket) \to (\Omega_F(\D(B)), d_F^\U)$ is a cochain map and $\Phi_\natural$ is surjective, it follows that the differential $\overline{d}_0 + \Theta_\flat$ must coincide with $d_F^\U$ as well.  Hence, we obtain the desired contraction
\[
 \begin{tikzcd}
\big(\D(F[1]), \llbracket d_F,- \rrbracket\big) \arrow[loop left, distance=2em, start anchor={[yshift=-1ex]west}, end anchor={[yshift=1ex]west}]{}{H_\natural} \arrow[r,yshift = 0.7ex, "\Phi_\natural"] & \big(\OmegaF(\D(B)) , d_F^\U\big) \arrow[l,yshift = -0.7ex, "\Psi_\natural"],
\end{tikzcd}
\]
where $\Psi_\natural = \Psi_\flat$ and $H_\natural = H_\flat$ are defined in~\eqref{Eq: Def of Psi and Hnatural}.
\end{proof}

\begin{remark}
  See~\cite{SVX} for an extension of Theorem~\ref{thm: Luca's result} to arbitrary Lie pairs.
\end{remark}

\subsection{Proof of Theorem~\ref{thm:contraction-on-DM}}\label{Sec: DO for perfect ID}
We now assume that $F$ is perfect and $B \subseteq T_\k M$ is an integrable distribution transversal to $F$. Thus, $(F,B)$ is a matched pair of Lie algebroids, and
  \begin{equation}\label{Eq: CD on F bowtie B}
    \begin{tikzcd}
      F \bowtie B \arrow[r] \arrow[d] & B \arrow[d] \\
             F           \arrow[r]               &  M
    \end{tikzcd}
  \end{equation}
  is a double Lie algebroid in the sense of Mackenzie~\cite{Mac}.
  According to Voronov~\cite{Voronov} (see also~\cite{SXsurvey}*{Theorem 3.3}), $F[1] \bowtie B \to F[1]$ is a dg Lie algebroid, where the dg manifold structures on $F[1] \bowtie B$ and $F[1]$ are induced from the horizontal Lie algebroid structures in~\eqref{Eq: CD on F bowtie B};
  according to Va\u{i}ntrob's theorem~\cite{Vaintrob}, the dg manifold structure on $F[1]$ is exactly $(F[1], d_F)$, while the one on $F[1] \bowtie B$ is induced from the dg module structure on $\Gamma(F[1], F[1] \bowtie B) \cong \Omega_F(B)$ arising from the Chevalley-Eilenberg differential $d_{F}^B$ of the Bott connection $\nabla^{\operatorname{Bott}}$ on $B$.
  Denote by $\mathfrak{B}$ the dg Lie algebroid $F[1] \bowtie B \to F[1]$.
  The anchor map
\[
\rho \colon F[1] \bowtie B \hookrightarrow T_{F[1]}
\]
coincides with the inclusion $\Psi$ in~\eqref{Eq: Psifork=1}.
Explicitly, the anchor map $\rho\colon \Omega_F(B) \to \Gamma(T_{F[1]}) \cong \Der(\Omega_F)$ is defined by
\begin{equation}\label{Eq: anchor map}
 \rho(\xi \otimes b) = \xi \otimes \widehat{b} \in \Gamma(T_{F[1]}) \cong \Der(\Omega_F),
\end{equation}
for any $\xi \in \Omega_F$  and $b \in \Gamma(B)$, where $\widehat{b} \in \Der(\Omega_F)$ is the horizontal lifting via the Bott $B$-connection $\eth$ on $F$ as in~\eqref{Eq: Bott B-connection}.
The Lie bracket on $\Gamma(\mathfrak{B}) = \Omega_F(B)$ is defined by
\[
 \{\xi_1 \otimes b_1, \xi_2 \otimes b_2\} = \xi_1 \wedge \eth_{b_1}\xi_2 \otimes b_2 - (-1)^{\abs{\xi_1}\abs{\xi_2}}\xi_2 \wedge \eth_{b_2} \xi_1 \otimes b_1 + \xi_1 \wedge \xi_2 \otimes [b_1,b_2],
\]
for any homogeneous $\xi_1, \xi_2 \in \Omega_F$ and $b_1,b_2 \in \Gamma(B)$, where $\eth$ is the flat $B$-connection on the vector bundle $\wedge F^\vee$ induced from the Bott $B$-connection $\eth$ on $F$, and $[-,-]$ is the Lie bracket on $\Gamma(B)$.

Consider the universal enveloping algebra $\U(\mathfrak{B})$ of the dg Lie algebroid $\mathfrak{B}$, which is a dg Hopf algebroid over $(\Omega_F, d_F)$. It is clear that $\U(\mathfrak{B}) \cong \Omega_F(\U(B))$ as dg Hopf algebroids, where the multiplication on $\Omega_F(\U(B))$ is defined in~\eqref{Eq: product}, and the comultiplication on $\Omega_F(\U(B))$ is the $\Omega_F$-linear extension of that of the Hopf algebroid $\U(B)$~\cite{Xu}.

Set
\[
  \mu \colon \Omega_F(\U(B)) (\cong \U(\mathfrak{B})) \to \D(F[1]) (\cong \U(T_{F[1]}))
\]
to be the morphism of dg algebras induced from the anchor map $\rho$ in~\eqref{Eq: anchor map}. Then $\mu$ is a morphism of dg Hopf algebroids over the dg algebra $(\Omega_F, d_F)$. Explicitly, one has
\begin{equation}\label{Eq: Psinatural of matched pairs}
  \mu(\xi \otimes b_1 b_2 \cdots b_n) = \xi \rho(b_1) \circ \rho(b_2) \cdots \circ \rho(b_n) = \xi \widehat{b_1} \circ \widehat{b_2} \circ \cdots \circ \widehat{b_n},
\end{equation}
for any $\xi \in \Omega_F$ and $b_1, \cdots, b_n \in \Gamma(B)$.

\begin{lemma}\label{Eq: CD in matched pair}
The following diagram
  \begin{equation*}
  \begin{tikzcd}
   \Gamma(ST_{F[1]}) \ar{r}{\pbw} & \D({F[1]})   \\
  \OmegaF(SB) \ar{r}{\overline{\pbw}} \ar{u}{\Psi}  & \OmegaF(\U(B))  \ar{u}{\mu}
  \end{tikzcd}
  \end{equation*}
commutes, where $\Psi \colon \Omega_F(SB) \to \Gamma(ST_{F[1]})$ is the inclusion map in the contraction~\eqref{Eq: contraction in CXX}. In other words, the map $\mu$ coincides with the inclusion map $\Psi_0$ defined in~\eqref{Eq: Def of phi0, psi0 and H0}.
\end{lemma}
\begin{proof}
Note that all maps in the above diagram are $\Omega_F$-linear. Hence, it suffices to show that
\begin{equation}\label{Eq: psi and upsi}
  \pbw \left(\Psi(b_1 \odot \cdots \odot b_n)\right) = \mu\left(\overline{\pbw}(b_1 \odot \cdots \odot b_n)\right),
\end{equation}
for any $n \geq 1$ and any $b_1, \cdots, b_n \in \Gamma(B)$. We prove this by induction.
When $n = 1$, one has
\[
\pbw \left( \Psi(b_1) \right) = \mu\left(\overline{\pbw}(b_1)\right) = \widehat{b_1}.
\]
Now assume that Equation~\eqref{Eq: psi and upsi} holds for some $n \geq 1$. Then we have
\begin{align*}
  & \quad\; \pbw\left(\Psi(b_1 \odot  \cdots \odot b_{n+1})\right) \\
   &= \pbw \left( \widehat{b_1} \odot \widehat{b_2} \odot \cdots \odot \widehat{b_{n+1}} \right) \\
    &= \frac{1}{n+1}\sum_{k=1}^{n+1}  \widehat{b_{k}} \circ \pbw \left( \widehat{b^{\{k\}}} \right) - \pbw\left(\pullbackconn_{\widehat{b_k}} \widehat{b^{\{k\}}}\right) \qquad\qquad \left(\text{by Equation~\eqref{Eq: Bconnection and pullback connection}}\right)\\
     &= \frac{1}{n+1}\sum_{k=1}^{n+1} \widehat{b_{k}} \circ \pbw\left( \Psi(b^{\{k\}})\right) - \pbw\left(\Psi (\nabla^B_{b_k} b^{\{k\}})\right)\quad \left(\text{by the inductive assumption}\right)\\
     &= \frac{1}{n+1}\sum_{k=1}^{n+1} \widehat{b_{k}} \circ \mu \left(\overline{\pbw} (b^{\{k\}})\right) - \pbw\left(\Psi  (\nabla^B_{b_k} b^{\{k\}} )\right) \\
      &= \frac{1}{n+1}\sum_{k=1}^{n+1}   \mu \left(b_{k} \circ \overline{\pbw} (b^{\{k\}} )\right) - \mu\left( \overline{\pbw}(\nabla^B_{b_k} b^{\{k\}})\right) \\
      &= \mu \left(\overline{\pbw}(b_1 \odot \cdots \odot b_{n+1})\right).
\end{align*}
Hence, Equation~\eqref{Eq: psi and upsi} holds for $n+1$, which concludes the proof.
\end{proof}

We are now ready to complete the proof of Theorem~\ref{thm:contraction-on-DM}.
\begin{proof}[Proof of Theorem~\ref{thm:contraction-on-DM}]
We prove that the inclusion $\Psi_\natural$ in Theorem~\ref{thm: Luca's result} coincides with the morphism $\mu$ as in~\eqref{Eq: Psinatural of matched pairs}.
We first prove that the inclusion map $\Psi_0$, the chain homotopy $H_0$ defined in~\eqref{Eq: Def of phi0, psi0 and H0} and the filtered perturbation $\Theta = \llbracket d_F,- \rrbracket - d_0 = \llbracket d_F,- \rrbracket - \pbw \circ L_{d_F} \circ \pbw^{-1}$ satisfy the following relation:
\[
H_0 \circ \Theta \circ \Psi_0 = 0 \colon \Omega_F(\U(B)) \to \D(F[1]).
\]
In fact, we have
\begin{align*}
   &\quad\; H_0  \circ \Theta \circ \Psi_0 \\
   &= H_0  \circ (\llbracket d_F,- \rrbracket - d_0)  \circ \Psi_0 \qquad \qquad \;\; \left(\text{since $\Psi_0$ is a cochain map}\right) \\
   &= H_0  \circ \llbracket d_F,- \rrbracket \circ \Psi_0  -  H_0 \circ \Psi_0 \circ \overline{d}_0 \quad \left(\text{by the side condition $H_0\Psi_0 = 0$ and Lemma~\ref{Eq: CD in matched pair}}\right) \\
   &= H_0 \circ \llbracket d_F,- \rrbracket \circ \mu \qquad \qquad \qquad\qquad  \left(\text{since $\mu$ is a cochain map}\right) \\
   &= H_0 \circ \mu \circ d^\U_F  \qquad \qquad \qquad\qquad\qquad  \left(\text{by the definition of $H_0$ in~\eqref{Eq: Def of phi0, psi0 and H0} and Lemma~\ref{Eq: CD in matched pair}}\right)\\
   &= \pbw \circ H \circ \pbw^{-1} \circ \pbw \circ \Psi \circ \overline{\pbw}^{-1} \circ d_F^\U \\
   &= \pbw\circ H \circ \Psi \circ \overline{\pbw}^{-1} \circ d_F^\U  \qquad\qquad \left(\text{by the side condition $H \Psi = 0$}\right)\\
   &= 0.
\end{align*}
 Hence, we have
\[
  \Psi_\natural := \sum_{k\geq 0}(H_0\Theta)^k \Psi_0 = \Psi_0 = \mu \colon \Omega_F(\U(B)) \to \D(F[1])
\]
by Lemma~\ref{Eq: CD in matched pair}. Thus, the inclusion $\Psi_\natural$ coincides with $\mu$, and therefore is a morphism of dg Hopf algebroids.
This completes the proof.
\end{proof}

\begin{remark}
 In general, the projection $\Phi_\natural$ is \textit{not} a morphism of dg Hopf algebroids, since it does not preserve the multiplications. For example, given   $x \in \Gamma(F)$ and $b \in \Gamma(B)$, under the identification~\eqref{Eqt:splitXFone}, one has $\widehat{x} \circ \widehat{b} \in \D(F[1])$  and
 \[
   \Phi_\natural\left(\widehat{x} \circ \widehat{b}\right) = \overline{x \circ b},
 \]
 whereas
 \[
  \Phi_\natural\big(\widehat{x}\big) \circ \Phi_\natural\Big( \widehat{b} \Big) = 0 \circ b = 0.
 \]
\end{remark}

\section{Duflo-Kontsevich type isomorphisms for integrable distributions}\label{Sec: main section}
This section is devoted to the proof of main theorems.
We start by recalling some existing results from~\cites{CXX,LSX,LSXpair}.
Note that, for the notations $\T_{\poly}^p(F[1])$ and $\D_{\poly}^p(F[1])$, we follow those in~\cites{CCN, SXsurvey}, which are also shifted by degree $(+1)$ comparing to~\cite{LSX}. Similarly,  the notations $\T_{\poly}^n(B)$ and $\D_{\poly}^n(B)$ in this paper are up to a degree shift comparing to~\cite{LSXpair}.
\subsection{Duflo-Kontsevich type isomorphism and Hochschild cohomology}
Throughout this section, we assume that $F \subseteq {\TkM}$ is an integrable distribution (not necessarily perfect).
\subsubsection{Atiyah and Todd classes arising from integrable distributions}
There are two types of Atiyah and Todd classes associated to an integrable distribution $F$, which are known to be isomorphic.
The first type is the Atiyah and Todd classes~\cite{MSX} (see also~\cite{LMS}) of the dg manifold $(F[1], d_F)$:
Given an  affine connection $\nabla$ on the graded manifold $F[1]$, consider the degree $(+1)$ map \[
\At_{(F[1],d_F)}^\nabla\colon \Gamma(T_{F[1]} \otimes T_{F[1]}) \to \Gamma(T_{F[1]})
\]
defined by
\[
 \At_{(F[1],d_F)}^\nabla(X,Y) = \llbracket d_F, \nabla_X Y \rrbracket - \nabla_{\llbracket d_F,X \rrbracket} Y - (-1)^{\abs{X}}\nabla_X \llbracket d_F,Y \rrbracket,
\]
for any $X,Y \in \Gamma(T_{F[1]})$.
It is easy to see that $\At_{(F[1],d_F)}^\nabla$ is $\Omega_F$-linear, and therefore is a bundle map $T_{F[1]} \otimes T_{F[1]} \to T_{F[1]}$, which can be identified with a degree $(+1)$ section of the graded vector bundle $T_{F[1]}^\vee \otimes \End(T_{F[1]})$.
It is also simple to check that $\At_{(F[1],d_F)}^\nabla$ is an $L_{d_F}$-cocycle, whose cohomology class
\[
 \left[\At_{(F[1],d_F)}^\nabla\right] \in \H^1\Big(\Gamma\big(T_{F[1]}^\vee \otimes \End(T_{F[1]})\big), L_{d_F}\Big)
\]
is independent of the choice of the connection $\nabla$ and is called the Atiyah class of the dg manifold $(F[1],d_F)$, denoted by $\At_{(F[1],d_F)}$ .
The Todd cocycle of the dg manifold $(F[1],d_F)$ associated with the affine connection $\nabla$ is
\[
 \Td_{(F[1],d_F)}^\nabla = \Ber\left(\frac{\At^\nabla_{(F[1],d_F)}}{1-e^{-\At^\nabla_{(F[1],d_F)}}} \right) \in \bigoplus_{k \geq 0} (\Omega^k(F[1]))^k.
\]
Its cohomology class $\Td_{(F[1],d_F)} \in \oplus_{k\geq 0} \H^k(\Omega^k(F[1]), L_{d_F})$ is independent of the choice of the connection $\nabla$, and is called the Todd class of the dg manifold $(F[1],d_F)$.

The second type is the Atiyah and Todd classes~\cite{CSX} of the Lie pair $({\TkM}, F)$. Let $\nabla^B$ be a ${\TkM}$-connection on the vector bundle $B = T_\k M/F$ extending the Bott $F$-connection. Consider the bundle map
\[
R_{1,1}^{\nabla^B}\colon F \otimes B \to \End(B)
\]
 defined by
\[
 R^{\nabla^B}_{1,1}(a, b) = \nabla^B_a\nabla^B_u - \nabla^B_u \nabla^B_a - \nabla^B_{[a,u]},
\]
for any $a \in \Gamma(F), u \in \Gamma({\TkM})$ satisfying $\pr_B(u) = b$.
One easily checks that $R^{\nabla^B}_{1,1} \in \Omega^1_F(B^\vee \otimes \End(B))$ is well-defined and is a $1$-cocycle of the Lie algebroid $F$ valued in the $F$-module $B^\vee \otimes \End(B)$. Its cohomology class $\At_{{\TkM}/F} \in \H^1_{\CE}(F, B^\vee \otimes \End(B))$ is independent of the choice of $\nabla^B$ and is called the Atiyah class of the Lie pair $({\TkM}, F)$.
The Todd cocycle of the Lie pair $({\TkM}, F)$ with respect to the chosen connection $\nabla^B$ is the Chevalley-Eilenberg cocycle
\[
  \Td_{{\TkM}/F}^{\nabla^B} := \det \left(\frac{R^{\nabla^B}_{1,1}}{1-e^{R^{\nabla^B}_{1,1}}}\right) \in \bigoplus_{k} \Omega_F^k(\wedge^k B^\vee),
\]
whose cohomology class $\Td_{{\TkM}/F} \in \bigoplus_{k} \H^k_{\CE}(F,\wedge^k B^\vee)$ is also independent of the choice of $\nabla^B$, and is called the Todd class of the Lie pair $({\TkM}, F)$.
The following proposition was proved in~\cite{CXX}.
\begin{proposition}[\cite{CXX}]\label{prop: Atiyah and Todd classes}
  There exist canonical isomorphisms
  \[
    \Phi \colon \H^k\big(\Omega^k(F[1]), L_{d_F}\big) \xrightarrow{\cong} \H^k_{\CE}\big(F, \wedge^k B^\vee\big),\quad k \geq 1,
  \]
  which send the Todd class of the dg manifold $(F[1],d_F)$ to that of the Lie pair $({\TkM}, F)$,  i.e.,
  \begin{align*}
    \Phi\left(\Td_{(F[1],d_F)}\right) = \Td_{{\TkM}/F}.
  \end{align*}
\end{proposition}

\subsubsection{Duflo-Kontsevich type isomorphisms for dg manifolds}\label{sec: kd for dg manifolds}
We now recall the Duflo-Kontsevich type isomorphism~\cites{LSX, SXsurvey} for the dg manifold $(F[1], d_F)$. Let $\T_{\poly}^p(F[1]) = \Gamma(S^p(T_{F[1]}[-1]))$ be the space of $p$-vector fields on the graded manifold $F[1]$.
The graded left $\Omega_F$-module
\[
 \T_{\poly}(F[1]) = \bigoplus_{p \in \Z^{\geq 0}} \T_{\poly}^p(F[1])
\]
is called the space of polyvector fields on $F[1]$. Let
\[
 \tot\big(\T_{\poly}(F[1])\big) = \bigoplus_n \big(\T_{\poly}(F[1])\big)^n
\]
be the associated \textit{direct sum} left graded $\Omega_F$-module. Here $\big(\T_{\poly}(F[1])\big)^n$ denotes the subspace consisting of all elements of degree $n$.
The graded commutator $[-,-]$ on $ \Gamma(T_{F[1]}) = \Der(\Omega_F)$ is a graded Lie bracket. It extends naturally to a degree $(-1)$ graded Lie bracket $[-,-]$, called the Schouten-Nijenhuis bracket, on $ \tot\big(\T_{\poly}(F[1])\big)$.
When equipped with Lie derivative $L_{d_F}$ along the homological vector field $d_F$, the quadruple $\big(\tot\big(\T_{\poly}(F[1])\big) , L_{d_F}, \wedge, [-,-]\big)$ is a dg Gerstenhaber algebra.
\begin{remark}
  Note that $\Omega_F$ as the space of the functions on the finite dimensional graded manifold $F[1]$ is non-negatively graded. Hence, $\T_{\poly}^1(F[1])$ is non-negatively graded as well.
  Thus, for polyvector fields on the dg manifold $(F[1],d_F)$,
  the direct sum total complex
  \[
  \left(\bigoplus_n \left(\bigoplus_{p \in \Z^{\geq 0}} \T_{\poly}^p(F[1])\right)^n, L_{d_F}\right)
  \]
   indeed coincides with the direct product total complex
   \[
  \left(\bigoplus_n \left(\prod_{p\in \Z^{\geq 0}}\T_{\poly}^p(F[1])\right)^n, L_{d_F}\right).
   \]
Therefore, there is no ambiguity for the notation $\tot\big(\T_{\poly}(F[1])\big)$.
\end{remark}
The space $\D_{\poly}^p(F[1])$ of $p$-differential operators on the graded manifold $F[1]$ is defined to be the tensor product $\otimes_{\Omega_F}^{p} (\D(F[1])[-1])$ of $p$-copies of the  $\Omega_F$-module $\D_{\poly}^1(F[1]):=\D(F[1])[-1]$.
The graded left $\Omega_F$-module
\[
\D_{\poly}(F[1]) = \bigoplus_{p \in \Z^{\geq 0}} \D_{\poly}^p(F[1])
\]
is called the space of polydifferential operators on $F[1]$. Let
\[
 \tot\big(\D_{\poly}(F[1])\big)  =\bigoplus_n  \big(\D_{\poly}(F[1])\big)^n
\]
be the associated \textit{direct sum} graded left $\Omega_F$-module.
Here we emphasize that there is a difference between taking direct sum $\bigoplus_{p \in \Z^{\geq 0}} \D_{\poly}^p(F[1])$ and direct product $\prod_{p \in \Z^{\geq 0}} \D_{\poly}^p(F[1])$ in the definition of $\D_{\poly}(F[1])$, since elements in $\D_{\poly}^1(F[1])$ may have negative degrees.

As in the classical case, the space $\tot(\D_{\poly}(F[1]))$ carries a standard Gerstenhaber bracket
\[
\llbracket -,- \rrbracket \colon \D^p_{\poly}(F[1]) \otimes_{\Omega_F} \D^{p^\prime}_{\poly}(F[1]) \rightarrow \D^{p+p^\prime-1}_{\poly}(F[1]).
\]
There are two differentials on this space, which make it into a double complex: one is the Lie derivative, or the Gerstenhaber bracket along $d_F$,
\[
 \llbracket d_F, - \rrbracket \colon \D^p_{\poly}(F[1]) \rightarrow \D^{p}_{\poly}(F[1])[1],
\]
and the other is the Hochschild differential
\[
 d_\mathscr{H} \colon  \D^p_{\poly}(F[1]) \rightarrow \D^{p+1}_{\poly}(F[1])[1]
\]
defined by
\begin{align*}
d_{\mathscr{H}}\big(\widetilde{D}_1 \otimes \cdots \otimes \widetilde{D}_p\big) &= 1 \otimes \widetilde{D}_1 \otimes \cdots \otimes \widetilde{D}_p + \sum_{i=1}^k(-1)^{\ast_i}  \widetilde{D}_1\otimes \cdots \otimes \Delta\big(\widetilde{D}_i\big)\otimes\cdots \otimes \widetilde{D}_p \\
  &\quad -(-1)^{\ast_p}\widetilde{D}_1 \otimes \cdots \otimes \widetilde{D}_p \otimes 1,
\end{align*}
where $\widetilde{D}_i\in \D^1_{\poly}(F[1]), \ast_i = \sum_{j=1}^i\abs{\widetilde{D}_j}$ for any $1 \leq i \leq p$. Note that the coproduct $\Delta\colon \D^1_{\poly}(F[1]) \to \D^2_{\poly}(F[1])$ stems from~\eqref{Eqt:DeltaDFone}. For details, see~\cite{CCN}.

The total differential $\llbracket d_F, - \rrbracket  + d_{\mathscr{H}}$ with the standard Gerstenhaber bracket $\llbracket -,- \rrbracket$ makes $\tot(\D_{\poly}(F[1]))$ into a dg Lie algebra of degree $(-1)$. Moreover, the tensor product of left $\Omega_F$-modules
\[
\cup \colon \D^p_{\poly}(F[1]) \times \D^{p^\prime}_{\poly}(F[1]) \rightarrow \D^{p+p^\prime}_{\poly}(F[1])
\]
determines a cup product on the Hochschild cohomology. It follows that the direct sum total cohomology
\[
 \Big(\H^\bullet\big(\tot(\D_{\poly}(F[1])), \llbracket d_F, - \rrbracket + d_{\mathscr{H}}\big), \llbracket -,- \rrbracket, \cup\Big)
\]
is a Gerstenhaber algebra.

The inclusion $\Gamma(T_{F[1]})[-1] \hookrightarrow \D(F[1])[-1]$ extends to a map
\[
\hkr \colon \tot(\T_{\poly}(F[1])) \to \tot(\D_{\poly}(F[1])),
\]
called the Hochschild-Kostant-Rosenberg map, and defined by
\begin{equation}\label{Eq: hkr for dg manifolds}
 \hkr(X_1 \odot \cdots \odot X_p) = \frac{1}{p!}\sum_{\sigma \in S_{p}} \kappa(\sigma; X_1, \cdots, X_p) X_{\sigma(1)} \otimes \cdots \otimes X_{\sigma(p)},
\end{equation}
for any homogeneous elements $X_1,\cdots, X_p \in \T^1_{\poly}(F[1])$,  where the  Koszul sign $\kappa(\sigma; X_1,\cdots, X_p)$ is defined by the relation $X_1 \odot \cdots \odot X_p = \kappa(\sigma; X_1,\cdots, X_p) X_{\sigma(1)} \odot \cdots \odot X_{\sigma(p)}$.

Applying the Duflo-Kontsevich type theorem~\cite{LSX}*{Theorem 4.3} (see also~\cite{SXsurvey}) for dg manifolds to this particular dg manifold $(F[1], d_F)$, we obtain
\begin{theorem}[\cite{LSX}]\label{prop: KD for dg manifold}
 The composition
\[
 \hkr \circ \Td_{(F[1],d_F)}^{1/2}\colon \H^\bullet\big(\tot\big(\T_{\poly}(F[1])\big), L_{d_F} \big) \xrightarrow{\cong} \H^\bullet\big(\tot\big(\D_{\poly}(F[1])\big), \llbracket d_F, - \rrbracket + d_{\mathscr{H}}\big)
\]
is an isomorphism of Gerstenhaber algebras, where $\Td_{(F[1],d_{F})}^{1/2} \in \oplus_k\H^k((\Omega^k(F[1]))^\bullet, L_{d_F})$ acts by contraction, and $\hkr$ is the Hochschild-Kostant-Rosenberg map~\eqref{Eq: hkr for dg manifolds}.
\end{theorem}
This isomorphism is called the Duflo-Kontsevich type isomorphism for the dg manifold $(F[1],d_F)$.
\begin{remark}
  Note that the Duflo-Kontsevich type isomorphism for dg manifolds is only valid for direct sum total cohomologies.
\end{remark}

\subsubsection{Cohomologies arising from the Lie pair $({\TkM}, F)$}\label{Secpolyofpair}
We now recall from~\cites{BSX,LSXpair} the cohomology of polyvector fields and that of polydifferential operators of the Lie pair $({\TkM}, F)$.
They can be thought of as polyvector fields and polydifferential operators on the leaf space of the foliation.
Let $\T_{\poly}^{0}(B)$ be the algebra $R$ of $\k$-valued smooth functions on $M$.
The space of polyvector fields of $({\TkM}, F)$ is a complex of $F$-modules with trivial differential
\[
\T_{\poly}(B) = \bigoplus_{n\geq 0} \T_{\poly}^n(B) = \bigoplus_{n\geq 0} \Gamma(\wedge^{n} B).
\]
By $\H^\bullet_{\CE}(F, \T_{\poly}(B))$, we denote the hypercohomology of the cochain complex
\[
 \left(\tot\big(\Omega_F(\T_{\poly}(B))\big) = \bigoplus_n\bigoplus_{p,q\geq0, p+q=n}\Gamma(\wedge^q F^\vee) \otimes_R \T^p_{\poly}(B), d_F^{\Bott}\right).
\]
Here the differential $d_F^{\Bott}$ is the Chevalley-Eilenberg differential induced from the obvious extension of the Bott $F$-connection on $B$, which is the leafwise de Rham differential with coefficient in $\T_{\poly}(B)$. Note that   we  count the  \textit{total}  degree for elements in  $\Omega_F(\T_{\poly}(B))$, i.e., elements in $\Gamma(\wedge^q F^\vee) \otimes_R \T^p_{\poly}(B)$ are of degree $p+q$.

Let $\D_{\poly}^{0}(B) = R$, and for each $k \geq 1$, $\D_{\poly}^{k}(B) = \otimes_R^k (\D(B))$ be the tensor product of $k$-copies of  the left $R$-module $\D(B) := \frac{\D(M)}{\D(M)\Gamma(F)}$.
Now we set
\[
 \D_{\poly}(B) = \bigoplus_{k \geq 0}\D_{\poly}^{k}(B).
\]
Since the comultiplication $\Delta \colon \D(B) \to \D(B) \otimes_R \D(B)$~\eqref{Eq: coproduct on DB} is coassociative,  the operator $d_{\mathscr{H}}\colon \D_{\poly}^{k}(B) \to \D_{\poly}^{k+1}(B)$ defined by
\begin{align*}
  d_{\mathscr{H}}(u_1 \otimes \cdots \otimes u_k) &= 1 \otimes u_1 \otimes \cdots \otimes u_k + \sum_{i=1}^k(-1)^iu_1\otimes\cdots\otimes\Delta(u_i)\otimes\cdots\otimes u_k \\
  &\quad - (-1)^{k}u_1 \otimes \cdots \otimes u_k \otimes 1,
\end{align*}
for any $u_1,\cdots,u_k \in \D(B)$, is of square zero, called the Hochschild differential.
Moreover, the comultiplication $\Delta$ is a morphism of $F$-modules. Hence, the Hochschild complex $(\D_{\poly}(B), d_{\mathscr{H}})$ is a complex of $F$-modules.
By $\H^\bullet_{\CE}\big(F, (\D_{\poly}(B), d_{\mathscr{H}})\big)$, we denote the hypercohomology of the cochain complex
\[
  \left(\tot\big(\Omega_F(\D_{\poly}(B))\big) = \bigoplus_n\bigoplus_{p, q \geq 0, p+q=n}\Gamma(\wedge^q F^\vee) \otimes_R \D^p_{\poly}(B), d_F^\U + \id \otimes d_{\mathscr{H}}\right),
\]
where $d_F^\U\colon \Omega_F^\bullet(\D_{\poly}(B)) \rightarrow \Omega_F^{\bullet+1}(\D_{\poly}(B))$ is the Chevalley-Eilenberg differential. Here again we count the  \textit{total}  degree for elements in $\Omega_F(\D_{\poly}(B))$, i.e., elements in $\Gamma(\wedge^q F^\vee) \otimes_R \D^p_{\poly}(B)$ are of degree $p+q$.

It is proved in~\cite{BSX} that both $\H^\bullet_{\CE}(F,\T_{\poly}(B))$ and $\H^\bullet_{\CE}\big(F, (\D_{\poly}(B), d_{\mathscr{H}})\big)$ carry canonical Gerstenhaber algebra structures, where the multiplications are wedge and cup products respectively, but the Lie brackets are much more involved and are obtained by homotopy transfer.

Note that the inclusion $\Gamma(B) \hookrightarrow \D(B)$ extends naturally by skew-symmetrization to a morphism of complex of $F$-modules $\hkr\colon (\T_{\poly}(B), 0) \to (\D_{\poly}(B), d_{\mathscr{H}})$, called the Hochschild-Kostant-Rosenberg map of the Lie pair $({\TkM}, F)$,
\begin{equation}\label{Eq: hkr for Lie pairs}
  \hkr(b_1\wedge \cdots \wedge b_n) = \frac{1}{n!}\sum_{\sigma \in S_n} \sgn(\sigma) b_{\sigma(1)} \otimes \cdots \otimes b_{\sigma(n)},\;\forall b_1,\cdots,b_n \in \Gamma(B).
\end{equation}
It is indeed a quasi-isomorphism of complexes of $F$-modules \cites{C-S-X,LSXpair}.
Therefore, it induces an isomorphism of vector spaces
\[
 \hkr \colon \H^\bullet_{\CE} \big(F, \T_{\poly}(B)\big) \stackrel{\cong}{\longrightarrow} \H^\bullet_{\CE}\big(F, (\D_{\poly}(B), d_{\mathscr{H}})\big).
\]
\subsubsection{Hochschild cohomology of integrable distributions}
We now can describe the Hochschild cohomology of the dg manifold $(F[1], d_F)$ by proving Theorem~\ref{Thm b} declared in the introduction, i.e., the following
\begin{theorem}\label{prop: contraction on Dpoly}
For any integrable distribution $F$, there is a contraction of dg $\Omega_F$-modules
	\begin{equation}\label{Eqt:ContractDoplyF1}
	\begin{tikzcd}
	\Big(\tot\big(\D_{\poly}({F[1]})\big), \llbracket d_F, - \rrbracket + d_\mathscr{H}\Big) \arrow[loop left, distance=2em, start anchor={[yshift=-1ex]west}, end anchor={[yshift=1ex]west}]{}{\breve{H}_\natural} \arrow[r,yshift = 0.7ex, "\Phi_\natural"] & \Big(\tot\big(\OmegaF(\D_{\poly}(B))\big) , \dF ^{\mathcal{U}} + \id \otimes d_\mathscr{H}\Big) \arrow[l,yshift = -0.7ex, "\breve{\Psi}_\natural"],
	\end{tikzcd}
	\end{equation}
where the projection $\Phi_\natural$ intertwines the associative products on $\tot(\D_{\poly}(F[1]))$ and $\tot(\Omega_F(\D_{\poly}(B)))$.
\end{theorem}
\begin{proof}
First, by a degree shifting on the contraction in Theorem~\ref{thm: Luca's result}, one has the following contraction
\[
\begin{tikzcd}
	\big(\D({F[1]})[-1], \llbracket d_F, - \rrbracket\big) \arrow[loop left, distance=2em, start anchor={[yshift=-1ex]west}, end anchor={[yshift=1ex]west}]{}{H_\natural} \arrow[r, yshift = 0.7ex, "\Phi_\natural"] & \big(\Omega_F(\D(B))[-1],  d_F^{\U}\big) \arrow[l, yshift = -0.7ex, "\Psi_\natural"].
	\end{tikzcd}
\]
Applying the tensor trick (see Lemma~\ref{Lem: Tensor trick}) to the above contraction,
we obtain the following contraction
	\begin{equation}\label{Eq: contraction dpoly before pur}
	\begin{tikzcd}
	\Big(\tot\big(\D_{\poly}(F[1])\big), \llbracket d_F,- \rrbracket\Big) \arrow[loop left, distance=2em, start anchor={[yshift=-1ex]west}, end anchor={[yshift=1ex]west}]{}{H_\natural} \arrow[r,yshift = 0.7ex, "\Phi_\natural"] & \Big(\tot\big(\OmegaF(\D_{\poly}(B))\big) , \dF^{\mathcal{U}}\Big) \arrow[l,yshift = -0.7ex, "\Psi_\natural"],
	\end{tikzcd}
	\end{equation}
where $\Phi_\natural$ and $\Psi_\natural$ are defined respectively by
\begin{align*}
  \Phi_\natural(\widetilde{D}_1 \otimes \cdots \otimes \widetilde{D}_n) &= \Phi_\natural(\widetilde{D}_1) \otimes \cdots \otimes \Phi_\natural(\widetilde{D}_n), \\
  \Psi_\natural(\widetilde{d}_1 \otimes \cdots \otimes \widetilde{d}_n) &= \Psi_\natural(\widetilde{d}_1) \otimes \cdots \otimes \Psi_\natural(\widetilde{d}_n),
\end{align*}
for any $\widetilde{D}_1,\cdots, \widetilde{D}_n \in \D(F[1])[-1] = \D^1_{\poly}(F[1]) $ and any $\widetilde{d}_1,\cdots, \widetilde{d}_n \in \OmegaF(\D(B))[-1]$, and the homotopy operator $H_\natural$ is defined by
\begin{align*}
& H_\natural(\widetilde{D}_1 \otimes \cdots \otimes \widetilde{D}_n) \\
:=& \sum_{k=1}^n \Psi_\natural(\Phi_\natural(\widetilde{D}_1)) \otimes \Psi_\natural(\Phi_\natural(\widetilde{D}_2))\otimes \cdots \otimes \Psi_\natural(\Phi_\natural (\widetilde{D}_{k-1}))\otimes  H_\natural(\widetilde{D}_k)  \otimes \widetilde{D}_{k+1} \otimes \cdots \otimes \widetilde{D}_{n}.
\end{align*}
For the space on the right-hand side in~\eqref{Eq: contraction dpoly before pur}, we have used the following natural identification
\[
 \otimes^p_{\Omega_F} \big(\Omega_F(\D(B))[-1]\big) \cong   \Omega_F\big(\otimes^p_R(\D(B))\big)[-p] = \Omega_F\big(\D^p_{\poly}(B)\big)[-p].
\]
To obtain the desired contraction~\eqref{Eqt:ContractDoplyF1}, we need to check that the perturbation $d_{\mathscr{H}}$ of the differential $\llbracket d_F,- \rrbracket$ on $\tot(\D_{\poly}(F[1]))$ satisfies the assumption~\eqref{Eq: perturb constraints} in the perturbation lemma~\ref{Lem: OPT}.

For this purpose, recall that $\D(F[1])$ (resp. $\D(B)$) admits an increasing filtration~\eqref{Eq: filtration on DF[1]} (resp.~\eqref{Eq: filtration on DB}).
Both $\Psi_\natural$ and $H_\natural$ in the contraction~\eqref{Contration:DFonetoOmegaFDpolyB} are filtered.
For any $\widetilde{d} \in \D^{\leq p}(B)[-1]$ considered as an element in $\OmegaF(\D^1_{\poly}(B))$ and $\widetilde{D} \in \D^{\leq p}(F[1])[-1]$ considered as an element in $\D^1_{\poly}(F[1])$, it follows from a direct computation that for any $n \geq p+1$,
\begin{align*}
  (H_\natural d_{\mathscr{H}})^n (\Psi_\natural(\widetilde{d})) = 0  &~\mbox{ ~and~ }~  (H_\natural d_{\mathscr{H}})^n (H_\natural(\widetilde{D}))   = 0.
\end{align*}
Using this fact, one obtains
\begin{align*}
  \cup_n \ker\big((H_\natural d_{\mathscr{H}})^n \Psi_\natural \big) &= \tot\big(\OmegaF(\D_{\poly}(B))\big)  & \mbox{and} \quad \cup_n \ker\big((H_\natural  d_{\mathscr{H}})^n H_\natural \big) &= \big( \tot \D_{\poly}(F[1]) \big).
\end{align*}
Meanwhile, since $\Phi_\natural$ in Theorem~\ref{thm: Luca's result} is a morphism of ${\OmegaF}$-coalgebras, it follows that $\Phi_\natural$ is compatible with the Hochschild differentials, i.e.,
\begin{align}\label{Eq: Phinatural and dH}
  \Phi_\natural  d_{\mathscr{H}} &= (\id \otimes d_{\mathscr{H}})  \Phi_\natural \colon \tot(\D_{\poly}(F[1])) \to \tot\big(\OmegaF(\D_{\poly}(B))\big).
\end{align}
Thus, using the side condition $\Phi_\natural H_\natural = 0$, one has
\begin{equation}\label{Eq: PhidHH}
  \Phi_\natural  d_{\mathscr{H}}  H_\natural = (\id \otimes d_{\mathscr{H}})  \Phi_\natural H_\natural = 0 \colon \tot(\D_{\poly}(F[1])) \to \tot\big(\OmegaF(\D_{\poly}(B))\big),
\end{equation}
which implies that
\[
 \cup_n \ker \big(\Phi_\natural (d_{\mathscr{H}} H_\natural)^n\big) = \tot\big(\D_{\poly}(F[1])\big).
\]
Thus, the perturbation $d_{\mathscr{H}}$ satisfies the constraint~\eqref{Eq: perturb constraints}.
Applying the perturbation lemma~\ref{Lem: OPT} to the contraction~\eqref{Eq: contraction dpoly before pur}, we obtain a new contraction
\[
	\begin{tikzcd}
	\Big(\tot\big(\D_{\poly}(F[1])\big), \llbracket d_F, - \rrbracket + d_{\mathscr{H}}\Big) \arrow[loop left, distance=2em, start anchor={[yshift=-1ex]west}, end anchor={[yshift=1ex]west}]{}{\breve{H}_\natural} \arrow[r,yshift = 0.7ex, "\breve{\Phi}_\natural"] & \Big(\tot\big(\OmegaF(\D_{\poly}(B))\big) , \dF ^{\mathcal{U}} + \varrho\Big) \arrow[l,yshift = -0.7ex, "\breve{\Psi}_\natural"],
	\end{tikzcd}
\]
where
\begin{align*}
 \varrho &= \sum_{n \geq 0}\Phi_\natural  (d_{\mathscr{H}}  H_\natural)^n d_{\mathscr{H}} \Psi_\natural \\
  &= \Phi_\natural  d_{\mathscr{H}} \Psi_\natural +  \sum_{n \geq 1}\Phi_\natural d_{\mathscr{H}}  H_\natural (d_{\mathscr{H}}  H_\natural)^{n-1} d_{\mathscr{H}}  \Psi_\natural  \qquad\qquad \left(\text{by Equation~\eqref{Eq: Phinatural and dH}} \right) \\
  &= \Phi_\natural d_{\mathscr{H}} \Psi_\natural + \sum_{n \geq 1} (\id \otimes d_{\mathscr{H}}) \Phi_\natural H_\natural  (d_{\mathscr{H}} H_\natural)^{n-1} d_{\mathscr{H}} \Psi_\natural   \qquad \left(\text{by the side condition $\Phi_\natural  H_\natural = 0$} \right) \\
 &= \Phi_\natural  d_{\mathscr{H}}  \Psi_\natural \\
 &= \id \otimes d_{\mathscr{H}},
 \end{align*}
and
\begin{align*}
  \breve{\Psi}_\natural &= \sum_{n \geq 0}(H_\natural  d_{\mathscr{H}})^n \Psi_\natural, &
  \breve{H}_{\natural} &= \sum_{n \geq 0}H_\natural  (d_{\mathscr{H}} H_\natural)^n.
\end{align*}
Moreover, one has
\begin{align*}
  \breve{\Phi}_\natural &= \sum_{n \geq 0}\Phi_\natural (d_{\mathscr{H}} H_\natural)^n = \Phi_\natural + \sum_{n \geq 1}\Phi_\natural d_{\mathscr{H}}  H_\natural  (d_{\mathscr{H}}  H_\natural)^{n-1} \qquad (\text{by Equation~\eqref{Eq: PhidHH}})\\
  &= \Phi_\natural.
\end{align*}
It is clear that $\Phi_\natural$ intertwines the associative products on $\tot(\D_{\poly}(F[1]))$ and $\tot\big(\Omega_F(\D_{\poly}(B))\big)$.
This completes the proof.
\end{proof}

As an immediate consequence, we  obtain the following
 \begin{corollary}\label{Coro: Isom of ass algebras on Dpoly}
The projection $\Phi_\natural$ in the contraction~\eqref{Eqt:ContractDoplyF1} induces an isomorphism of associative algebras on the cohomology
\begin{equation}\label{eq:SCE}	
\Phi_\natural \colon \H^\bullet\big(\tot\big(\D_{\poly}({F[1]})\big), \llbracket d_F, - \rrbracket  + d_{\mathscr{H}}\big) \xrightarrow{\cong} \H_{\CE}^\bullet\big(F, (\D_{\poly}(B), d_{\mathscr{H}})\big).
\end{equation}	
\end{corollary}

\begin{remark}
  In particular, when $F = 0$, both sides of~\eqref{Eqt:ContractDoplyF1} become the cochain complex $(\D_{\poly}(M), d_{\mathscr{H}})$ of polydifferential operators on $M$. Hence, both sides of \eqref{eq:SCE} become  the smooth Hochschild cohomology of $C^\infty(M)$.

  On the other hand, when $F = T_\k M$, the commutative dg algebra $(\Omega_F,d_F)$ becomes the de Rham dg algebra $(\Omega_M, d_{\operatorname{dR}})$, and the normal bundle $B$ is the rank zero vector bundle over $M$.
The right hand side of \eqref{eq:SCE} is simply the de Rham cohomology of $M$. Hence, by Corollary~\ref{Coro: Isom of ass algebras on Dpoly}, we obtain an isomorphism
  \[
   \Phi_\natural \colon  \H^\bullet\big(\tot\big(\D_{\poly}({T_\k [1]M})\big), \llbracket d_{\operatorname{dR}}, - \rrbracket  + d_{\mathscr{H}}\big) \xrightarrow{\cong} \H_{\operatorname{dR}}^\bullet(M).
  \]
  Namely, the Hochschild cohomology of the dg algebra $(\Omega_M, d_{\operatorname{dR}})$,
 defined as the direct sum total cohomology of the double complex $\big(\D_{\poly}({T_\k [1]M}), \llbracket d_{\operatorname{dR}}, - \rrbracket  + d_{\mathscr{H}}\big)$, is isomorphic to the de Rham cohomology of $M$. This statement is false if we use the ordinary Hochschild cohomology of the dg algebra $(\Omega_M, d_{\operatorname{dR}})$, i.e., the direct product total cohomology.  See~\cites{Chen, CFL, GJP, RWang} for details.
\end{remark}

\subsubsection{Proof of Theorem~\ref{Main thm}}
We are now ready to prove Theorems~\ref{Main thm} declared in Introduction. First let us recall the following
\begin{proposition}[\cite{CXX}*{Corollary 2.43}]\label{prop: isom on Tpoly}
  Let $F \subseteq {\TkM}$ be an integrable distribution. There is a canonical isomorphism of Gerstenhaber algebras
  \[
  \Phi \colon  \H^\bullet\big(\tot\big(\T_{\poly}(F[1])\big), L_{d_F}\big) \xrightarrow{\cong} \H^\bullet_{\CE}\big(F, \T_{\poly}(B)\big)
  \]
  from the cohomology of polyvector fields on the dg manifold $(F[1],d_F)$ to the Chevalley-Eilenberg hypercohomology of polyvector fields of the Lie pair $({\TkM}, F)$.
\end{proposition}

\begin{theorem}\label{Thm: KD for integrable distributions}
Let $F \subseteq {\TkM}$ be an integrable distribution. There is a commutative diagram:
  \[
   \begin{tikzcd}
   \H^\bullet\Big(\tot\big(\mathcal{T}_{\poly}({F[1]})\big), L_{d_F}\Big)  \ar{d}[left]{\Phi}[right]{\cong} \ar{rrr}{\hkr \circ \Td^{1/2}_{(F[1],d_F)}}[below]{\cong} &&& \H^\bullet\Big(\tot\big(\D_{\poly}({F[1]})\big), \llbracket d_F, - \rrbracket + d_{\mathscr{H}}\Big) \ar{d}[left]{\Phi_\natural}[right]{\cong} \\
   \H^\bullet_{\CE}\big(F,\mathcal{T}_{\poly}(B)\big)  \ar{rrr}{\hkr \circ \Td^{1/2}_{{\TkM}/F}}[below]{\cong} &&& \H^\bullet_{\CE}\big( F,(\D_{\poly}(B), d_{\mathscr{H}}) \big).
   \end{tikzcd}
  \]
\end{theorem}
\begin{proof}		
Since $\Phi_\natural\mid_{\D^{\leq 1}({F[1]})} = \Phi$ and the two types of Hochschild-Kostant-Rosenberg isomorphisms $\hkr$ for dg manifolds~\eqref{Eq: hkr for dg manifolds} and for Lie pairs~\eqref{Eq: hkr for Lie pairs} are defined by (skew-)symmetrization,
the projections $\Phi$ in Proposition~\ref{prop: isom on Tpoly} and $\Phi_\natural$ in Corollary~\ref{Coro: Isom of ass algebras on Dpoly} are compatible with the two isomorphisms $\hkr$, i.e., the following diagram commutes
  \[
   \begin{tikzcd}
   \H^\bullet\Big(\tot\big(\mathcal{T}_{\poly}({F[1]})\big), L_{d_F} \Big) \ar{d}[left]{\Phi} \ar{rr}{\hkr} && \H^\bullet\Big(\tot\big(\D_{\poly}({F[1]})\big), \llbracket d_F, - \rrbracket  + d_{\mathscr{H}}\Big) \ar{d}[left]{\Phi_\natural} \\
   \H^\bullet_{\CE}\big( F,\mathcal{T}_{\poly}(B) \big) \ar{rr}{\hkr} && \H^\bullet_{\CE}\big(F, (\D_{\poly}(B), d_{\mathscr{H}})\big).
   \end{tikzcd}
  \]
By Proposition~\ref{prop: Atiyah and Todd classes}, the projection $\Phi$ sends the Todd class $\Td_{(F[1], d_F)}$ of the dg manifold $(F[1],d_F)$ to the Todd class $\Td_{{\TkM}/F}$ of the Lie pair $({\TkM}, F)$.
Thus, the contraction operators by the two Todd classes are compatible with the projection $\Phi$, i.e., the following diagram commutes
  \[
   \begin{tikzcd}
   \H^\bullet\Big(\tot\big(\mathcal{T}_{\poly}({F[1]})\big), L_{d_F}\Big) \ar{d}[left]{\Phi} \ar{rr}{\Td^{1/2}_{({F[1]},d_F)}} && \H^\bullet\Big(\tot\big(\mathcal{T}_{\poly}({F[1]})\big), L_{d_F}\Big) \ar{d}[left]{\Phi} \\
   \H^\bullet_{\CE}\big(F,\mathcal{T}_{\poly}(B) \big)  \ar{rr}{  \Td^{1/2}_{{\TkM}/F}} && \H^\bullet_{\CE}\big(F,\mathcal{T}_{\poly}(B) \big).
   \end{tikzcd}
  \]
Combining the above two commutative diagrams, we conclude the proof.
\end{proof}

\subsection{Isomorphisms of Gerstenhaber algebras for perfect integrable distributions}
We now assume that $F$ is perfect. As a direct consequence of Theorems~\ref{thm:contraction-on-DM} and~\ref{prop: contraction on Dpoly}, we obtain the following contraction
\begin{equation}\label{prop: contraction of Dpoly for perfect ID}
\begin{tikzcd}
	\Big(\tot\big(\D_{\poly}({F[1]})\big), \llbracket d_F, - \rrbracket   + d_\mathscr{H}\Big) \arrow[loop left, distance=2em, start anchor={[yshift=-1ex]west}, end anchor={[yshift=1ex]west}]{}{\breve{H}_\natural} \arrow[r,yshift = 0.7ex, "\Phi_\natural"] & \Big(\tot\big(\OmegaF(\U_{\poly}(B))\big) , \dF ^{\mathcal{U}} + \id \otimes d_\mathscr{H}\Big) \arrow[l,yshift = -0.7ex, "\Psi_\natural"],
\end{tikzcd}
\end{equation}
where
\[
\U_{\poly}(B) = \bigoplus_{k \geq 0}\U^k_{\poly}(B)=\bigoplus_{k \geq 0}\otimes^k_R (\U(B)).
\]
Here we have used the assumption that $F$ is perfect, which implies the equality
\begin{align*}
 \breve{\Psi}_\natural &=  \sum_{n \geq 0}(H_\natural d_{\mathscr{H}})^n \Psi_\natural \\
  &= \Psi_\natural + \sum_{n\geq 1}(H_\natural d_{\mathscr{H}})^{n-1} H_\natural  d_{\mathscr{H}} \Psi_\natural\qquad \left(\text{by Theorem~\ref{thm:contraction-on-DM}}\right)\\
   &= \Psi_\natural + \sum_{n\geq 1}(H_\natural d_{\mathscr{H}})^{n-1}  H_\natural  \Psi_\natural  d_{\mathscr{H}} \qquad \left(\text{by the side condition $H_\natural \Psi_\natural = 0$}\right)\\
  &= \Psi_\natural.
\end{align*}

Recall that by $\mathfrak{B}$ we denote the dg Lie algebroid $F[1] \bowtie B \to F[1]$ (see Section \ref{Sec: DO for perfect ID}). The space $\Omega_F(\U(B))$, which is isomorphic to the universal enveloping algebra $\U(\mathfrak{B})$ of the dg Lie algebroid $\mathfrak{B}$,  carries a dg Hopf algebroid structure over $(\Omega_F, d_F)$.
Note that we have a natural isomorphism
\[
\tot\big(\Omega_F(\U_{\poly}(B))\big) \cong \tot\big(\oplus_{k\geq 0}\otimes_{\Omega_F}^{k}(\U(\mathfrak{B})[-1])\big).
\]
Thus the Gerstenhaber bracket on the right-hand side $\tot\big(\oplus_{k\geq 0} \otimes_{\Omega_F}^{k} (\U(\mathfrak{B})[-1])\big)$ induces a Gerstenhaber bracket on the left-hand side $\tot\big(\Omega_F(\U_{\poly}(B))\big)$, which can be expressed explicitly as follows: For any homogeneous $D \in \Omega_F(\U^u_{\poly}(B)), D^\prime \in \Omega_F(\U^v_{\poly}(B))$,
\begin{equation}\label{Eq: Def of Ger bracket}
\llbracket D, D^\prime \rrbracket := D \star D^\prime - (-1)^{(\abs{D}  -1)(\abs{D^\prime}  -1)} D^\prime \star D \in \Omega_F(\U^{u+v-1}_{\poly}(B) ),
\end{equation}
where
\begin{equation}\label{Eq: star product}
D \star D^\prime := \sum_{k=1}^u(-1)^{(\abs{D^\prime} -1)\dagger_k} d_1 \otimes \cdots \otimes d_{k-1} \otimes (\Delta^{v-1} d_k)\cdot D^\prime \otimes d_{k+1} \otimes \cdots \otimes d_{u},
\end{equation}
for any $D = d_1 \otimes \cdots \otimes d_u$ with homogeneous $d_1,\cdots,d_u \in \Omega_F( \U(B))$. Here $\dagger_k$ is defined to be $\abs{d_{k+1}}+\cdots+\abs{d_u}$ for any $1 \leq k \leq u$.
To understand the product $(\Delta^{v-1} d_k) \cdot D^\prime$ in $\Omega_F( \U^v_{\poly}(B))$ appeared in the above equation, one needs the compatibility axiom between the product and coproduct of the Hopf algebroid $\Omega_F(\U(B))$ over $(\Omega_F, d_F)$, for which we refer the reader to~\cite{Xu} for details.
We remind the reader that this Gerstenhaber bracket is \emph{not} the $\Omega_F$-linear extension of the Gerstenhaber bracket on $\U_{\poly}(B)$, since the product on $\Omega_F(\U(B))$ is \emph{not} $\Omega_F$-linear.

\begin{theorem}\label{prop: Isom of G-algebra of Dpoly for perfect ID}
  Let $F \subseteq {\TkM}$ be a perfect integrable distribution.
  \begin{compactenum}
  	\item The map $\Psi_\natural\colon \tot\big(\OmegaF(\U_{\poly}(B))\big) \to \tot(\D_{\poly}({F[1]}))$   as in~\eqref{prop: contraction of Dpoly for perfect ID} preserves the Gerstenhaber brackets:
  	\begin{equation*}
  	\Psi_{\natural}(\llbracket D,D^\prime \rrbracket) = \llbracket \Psi_{\natural}(D), \Psi_{\natural}(D^\prime) \rrbracket,\qquad \forall D,D'\in \tot\big(\OmegaF(\U_{\poly}(B))\big).
  	\end{equation*}
  	\item Passing to the cohomologies, $\Phi_\natural$ and $\Psi_\natural$  as in~\eqref{prop: contraction of Dpoly for perfect ID}   are isomorphisms of Gerstenhaber algebras, which are mutually inverse to each other:
  \[
  \begin{tikzcd}
 \H^\bullet\Big(\tot\big(\D_{\poly}({F[1]})\big), \llbracket d_F, - \rrbracket + d_{\mathscr{H}}\Big) \arrow[r,yshift = 0.7ex, "\Phi_\natural"] & \H^\bullet_{\CE}\Big(F, \big(\U_{\poly}(B), d_{\mathscr{H}}\big)\Big) \arrow[l,yshift = -0.7ex, "\Psi_\natural"] .
  \end{tikzcd}
  \]
  \end{compactenum}
\end{theorem}
\begin{proof}
 It suffices to prove the first statement.
Note that the map $\Psi_\natural \colon \Omega_F(\U(B)) \to \D(F[1])$ preserves both multiplications and comultiplications according to Theorem~\ref{thm:contraction-on-DM}. Thus, it follows that for any $D = d_1 \otimes \cdots \otimes d_u \in \Omega_F(\U^u_{\poly}(B)) $ with homogeneous $d_1,\cdots,d_u \in \Omega_F(\U^1_{\poly}(B))$ and  $D^\prime \in \Omega_F(\U^v_{\poly}(B))$, we have
  \begin{align*}
    \Psi_\natural (D \star D^\prime) &= \sum_{k=1}^u(-1)^{\dagger_k(\abs{D^\prime} -1)} \Psi_{\natural}(d_1) \otimes \cdots \otimes \Psi_{\natural}((\Delta^{v-1} d_k)\cdot D^\prime) \otimes \cdots \otimes \Psi_{\natural}(d_{u}) \\
    &=  \sum_{k=1}^u(-1)^{\dagger_k(\abs{D^\prime} -1)}\Psi_{\natural}(d_1) \otimes \cdots \otimes \Delta^{v-1}(\Psi_{\natural}(d_k))\cdot \Psi_\natural(D^\prime) \otimes \cdots \otimes \Psi_{\natural}(d_{u}) \\
    &= \Psi_{\natural}(D) \star \Psi_{\natural}(D^\prime).
  \end{align*}
Hence, we have
\[
  \Psi_{\natural}(\llbracket D,D^\prime \rrbracket) = \llbracket \Psi_{\natural}(D), \Psi_{\natural}(D^\prime) \rrbracket.
\]
That is, $\Psi_\natural$ is a morphism of Gerstenhaber algebras.
\end{proof}

\begin{remark}\label{Main remark}
In general, without the perfect assumption on $F$,  the space $\Omega_F(\D(B))$ of differential operators on the Lie pair $({\TkM}, F)$ does \textit{not} admit an associative algebra structure.
Indeed, it was proved by Vitagliano in~\cite{Luca} that  $\Omega_F(\D(B))$ admits an $A_\infty$-algebra structure. One \textit{cannot} define a Gerstenhaber algebra structure directly on the total cohomology $\H_{\CE}^\bullet(F, (\D_{\poly}(B), d_{\mathscr{H}}))$ using Equations~\eqref{Eq: Def of Ger bracket} and~\eqref{Eq: star product}.
However, Bandiera, Sti\'{e}non and Xu proved in~\cite{BSX} that there exists  a canonical Gerstenhaber algebra structure on $\H_{\CE}^\bullet(F, (\D_{\poly}(B), d_{\mathscr{H}}))$ by applying the homotopy transfer theorem to the Dolgushev-Fedosov contraction for polydifferential operators on Lie pairs. When endowed with this Gerstenhaber algebra structure, we expect that
 \[
  \Phi_\natural \colon  \H^\bullet\big(\tot\big(\D_{\poly}({F[1]})\big), \llbracket d_F, - \rrbracket  + d_{\mathscr{H}}\big) \xrightarrow{\cong} \H_{\CE}^\bullet\big(F, \big(\D_{\poly}(B), d_{\mathscr{H}}\big)\big)
  \]
 is still an isomorphism of Gerstenhaber algebras. We would like to return to this question in the future.
\end{remark}

\subsection{Application to complex manifolds}
As an application, consider a complex manifold $X$.
The subbundle $F = T_X^{0,1} \subset T_\C X$ is a perfect integrable distribution, and the quotient bundle $B := T_\C X/T_X^{0,1}$ is naturally identified with $T^{1,0}_X$. Moreover, the Chevalley-Eilenberg differential associated with the Bott $F$-connection on $T_X^{1,0}$ becomes the Dolbeault operator
\[
 \bar{\partial}\colon \Omega_X^{0,\bullet}(T_X^{1,0}) \rightarrow \Omega_X^{0,\bullet+1}(T_X^{1,0}).
\]
In this setting, the space $\T_{\poly}(B)$ of polyvector fields of the Lie pair $(T_\C X, T_X^{0,1})$ coincides with the space $\wedge T_X^{1,0}$. The cochain complex $\big(\tot\left(\Omega_F(\mathcal{T}_{\poly}(B)) \right), d_{F}^{\operatorname{Bott}}\big)$ becomes  $\big(\tot\big(\Omega_X^{0,\bullet} (\mathcal{T}_{\poly}(X))\big), \bar{\partial}\big)$, which is indeed the Dolbeault resolution of the complex of sheaves of $\O_X$-modules
\[
  0 \to \O_X \xrightarrow{0} \mathcal{T}_{\poly}^1(X) \xrightarrow{0} \mathcal{T}_{\poly}^2(X) \xrightarrow{0} \mathcal{T}^3_{\poly}(X) \to \cdots.
\]
Thus, the cohomology of the complex $\big(\tot\left(\Omega_F(\mathcal{T}_{\poly}(B))\right), d_F^{\Bott}\big)$ is isomorphic to the sheaf cohomology of $\mathcal{T}_{\poly}(X)$, i.e.,
\[
\H^\bullet_{\CE}(F, \T_{\poly}(B)) \cong \H^\bullet(X, \mathcal{T}_{\poly}(X)).
\]
On the other hand, by the canonical identification
\[
 \D(B) = \frac{\U(T_\C X)}{\U(T_\C X)\Gamma(T_X^{0,1})} \cong \U(T_X^{1,0}),
\]
the cochain complex $\big(\tot\big(\Omega_F(\D_{\poly}(B))\big), d_F^\U + \id \otimes d_{\mathscr{H}}\big)$ becomes $\big(\tot\big(\Omega_X^{0,\bullet}(\mathscr{D}_{\poly}(X))\big), \bar{\partial} + \id \otimes d_{\mathscr{H}}\big)$, that is, the Dolbeault resolution of the complex of sheaves
\[
 0 \rightarrow \O_X \rightarrow \mathscr{D}^1_{\poly}(X)\xrightarrow{d_{\mathscr{H}}} \mathscr{D}^2_{\poly}(X)  \xrightarrow{d_{\mathscr{H}}} \mathscr{D}^3_{\poly}(X) \rightarrow \cdots
\]
of holomorphic polydifferential operators over $X$. Its total cohomology is isomorphic to the Hochschild cohomology of the complex manifold $X$ (cf.~\cites{Caldararu, Yekutieli}), i.e.,
\[
 \H^\bullet_{\CE}\big(F, (\D_{\poly}(B), d_{\mathscr{H}})\big) \cong \hochschildcohomology{\bullet}(X).
\]
Applying Theorem~\ref{Thm: KD for integrable distributions} to the perfect integrable distribution $T_X^{0,1} \subset T_\mathbb{C} X$,  we obtain the following %  (Theorem \ref{Thm B})
\begin{theorem}\label{Thm: Complex manifolds}
  Let $(T_X^{0,1}[1],\bar{\partial})$ be the dg manifold arising from a complex manifold $X$. We have the following commutative diagram of cohomology groups
  \[
   \begin{tikzcd}
   \H^\bullet\Big(\tot\big(\mathcal{T}_{\poly}(T_X^{0,1}[1])\big), L_{\bar{\partial}}\Big) \ar{d}[left]{\Phi}[right]{\cong} \ar{rrr}{\hkr \circ \Td^{1/2}_{(T^{0,1}_X[1], \bar{\partial})}} &&& \H^\bullet\Big(\tot\big(\D_{\poly}(T^{0,1}_X[1])\big), \llbracket \bar{\partial}, - \rrbracket  + d_{\mathscr{H}}\Big) \ar{d}[right]{\Phi_\natural}[left]{\cong} \\
   \H^\bullet\big(X, \mathcal{T}_{\poly}(X)\big) \ar{rrr}{{\hkr}~ \circ~ {\Td^{1/2}_{T_\C X/T_X^{0,1}}}} &&& \hochschildcohomology{\bullet}(X).
   \end{tikzcd}
  \]
\end{theorem}
From this theorem, we conclude that the Duflo-Kontsevich theorem for complex manifolds~\cites{Kon, CV} (see also \cite{LSXpair}) is a direct consequence of the Duflo-Kontsevich type isomorphism (Theorem~\ref{prop: KD for dg manifold}) for the dg manifold $(T_X^{0,1}[1],\bar{\partial})$.
\begin{theorem}\label{main corollary}
  For every complex manifold $X$, the composition
  \begin{equation}\label{Eq: KD isomorphism for complex manifolds}
   \hkr \circ  \Td_{T_\mathbb{C} X/T_X^{0,1}} ^{1/2}\colon \H^\bullet\big(X, \mathcal{T}_{\poly}(X)\big) \xrightarrow{\cong} \hochschildcohomology{\bullet}(X)
  \end{equation}
  is an isomorphism of Gerstenhaber algebras, where the square root of the Todd class
  \[
    \Td_{T_\mathbb{C} X/T_X^{0,1}} \in \bigoplus_{k \geq 0} \H^k(X, \Omega_X^k)
  \]
  acts on $\H^\bullet(X, \mathcal{T}_{\poly}(X))$ by contraction.
\end{theorem}
\begin{proof}
       By the Duflo-Kontsevich type isomorphism in Theorem~\ref{prop: KD for dg manifold} for the dg manifold $(T_X^{0,1}[1], \bar{\partial})$,  the map
       \[
       \hkr \circ \Td^{1/2}_{(T_X^{0,1}[1], \bar{\partial})} \colon \H^\bullet\big(\tot\big(\mathcal{T}_{\poly}(T_X^{0,1}[1])\big), L_{\bar{\partial}}\big) \xrightarrow{\cong}  \H^\bullet\big(\tot\big(\D_{\poly}(T^{0,1}_X[1])\big), \llbracket \bar{\partial}, - \rrbracket  + d_{\mathscr{H}}\big)
       \]
       is an isomorphism of Gerstenhaber algebras.
       Applying Proposition~\ref{prop: isom on Tpoly} and Theorem~\ref{prop: Isom of G-algebra of Dpoly for perfect ID}  to the corresponding perfect integrable distribution $T_X^{0,1} \subset T_\C X$, we see that both
       \[
        \Phi \colon  \H^\bullet\big(\tot\big(\mathcal{T}_{\poly}(T_X^{0,1}[1])\big), L_{\bar{\partial}}\big) \xrightarrow{\cong}  \H^\bullet\big(X, \mathcal{T}_{\poly}(X)\big),
       \]
       and
       \[
        \Phi_\natural \colon \H^\bullet\big(\tot\big(\D_{\poly}(T^{0,1}_X[1])\big), \llbracket \bar{\partial}, - \rrbracket  + d_{\mathscr{H}}\big) \xrightarrow{\cong} \hochschildcohomology{\bullet}(X)
       \]
       are isomorphisms of Gerstenhaber algebras.
       Now the commutative diagram in Theorem~\ref{Thm: Complex manifolds} implies that the map in~\eqref{Eq: KD isomorphism for complex manifolds} must be an isomorphism of Gerstenhaber algebras as well.
\end{proof}

\begin{remark}
  The Duflo-Kontsevich theorem for complex manifolds is due to Kontsevich~\cite{Kon}---where only the associative algebra structures were addressed. Calaque and Van den Bergh proved the isomorphism of Gerstenhaber algebras for any smooth algebraic variety $X$ in~\cite{CV};
  Liao, Sti\'{e}non and Xu gave a different proof for any complex manifold in~\cite{LSXpair} via formality for Lie pairs.
Note that the Todd class $\Td_{T_{\mathbb{C}}X/T_X^{0,1}}$ of the Lie pair $(T_{\mathbb{C}}X, T_X^{0,1})$ coincides with the Todd class $\Td_X$ of $X$, when $X$ is a  compact K\"{a}hler manifold, or is algebraic and proper.
\end{remark}

\appendix
\section{The homological perturbation lemma}
A contraction of  cochain ($\k$-)complexes $(P,\delta)$ onto  $(T,d)$ consists of $\k$-linear maps $\phi$, $\psi$, and $h$ symbolized by a diagram
\begin{equation}\label{Eqt:contractionDatum}
\begin{tikzcd}
(P, \delta) \arrow[loop left, distance=2em, start anchor={[yshift=-1ex]west}, end anchor={[yshift=1ex]west}]{}{h } \arrow[r,yshift = 0.7ex, "\phi"] & (T , d) \arrow[l,yshift = -0.7ex, "\psi"],
\end{tikzcd}
\end{equation}
where $\phi$ and $\psi$ are cochain maps  and $h \colon P \to P$ is of degree $(-1)$, satisfying the homotopy retraction relations
\begin{align*}
\phi \psi &=\id_{T}, & \psi  \phi &=\id_P+h  \delta+\delta h,
\end{align*}
and the side conditions
\begin{align*}
\phi  h &=0, & h \psi &=0, &  h^2 &= 0.
\end{align*}
A \textit{perturbation} of the differential $\delta$ is a linear map $\varrho\colon P \to P[1]$ such that $\delta+\varrho$ is a new differential on $P$.
The following perturbation lemma is standard. See~\cites{Brown, HK}.
\begin{lemma}\label{Lem: OPT}
  Assume that the perturbation $\varrho$ satisfies the following constraints
  \begin{align}\label{Eq: perturb constraints}
    \cup_n \ker((h\varrho)^n\psi) &= T, & \cup_n \ker(\phi (\varrho  h)^n) &= P, & \cup_n \ker(h (\varrho h)^n) &= P.
  \end{align}
  Then the series
\begin{align}
	\vartheta &:=\sum_{k=0}^\infty \phi (h\varrho)^k\varrho \psi, & 	\phi_\flat &:=\sum_{k=0}^\infty \phi (\varrho h)^k, \label{Eq: Def of thetaandphi}\\
	\psi_\flat &:=\sum_{k=0}^\infty (h \varrho)^k \psi, & h_\flat &:=\sum_{k=0}^\infty h (\varrho h)^k \label{Eq: Def of psiandh}
\end{align}
all converge, and the datum
\[
	\begin{tikzcd}
	(P, \delta  +\varrho) \arrow[loop left, distance=2em, start anchor={[yshift=-1ex]west}, end anchor={[yshift=1ex]west}]{}{h_\flat } \arrow[r,yshift = 0.7ex, "\phi_\flat"] & (T ,  d+\vartheta) \arrow[l,yshift = -0.7ex, "\psi_\flat"]
	\end{tikzcd}
\]
constitutes a new contraction.
\end{lemma}
A particular class arises from perturbation of filtered complexes.
 Suppose that the contraction~\eqref{Eqt:contractionDatum} is increasingly filtered (cf.~\cite{EM}), that is, $P, T$ are increasingly filtered, and the maps $\phi,\psi, h$ preserve the filtrations.
An increasing filtration on a cochain complex $P$
\[
\cdots \subseteq \mathcal{F}^{n-1}P \subseteq \mathcal{F}^nP \subseteq \mathcal{F}^{n+1}P \subset \cdots
\]
 is said to be \emph{exhaustive} if $P = \cup_{n} \mathcal{F}^n P$ and \emph{bounded from below} if there exists an integer $m$ such that $\mathcal{F}^kP = 0$ for any $k \leq m$.

 Assume further that the filtration of $P$ and $T$ in the contraction~\eqref{Eqt:contractionDatum} are exhaustive and bounded below. If the perturbation $\varrho$ of the differential $\delta$ on $P$ lowers the filtration by $1$, that is,  $\varrho(\mathcal{F}^{n}P) \subseteq \mathcal{F}^{n-1}P$, then it is clear that all constraints in~\eqref{Eq: perturb constraints} hold.
 Applying Lemma~\ref{Lem: OPT}, one has the following filtered perturbation lemma. See~\cites{Brown, Manettibook}.
\begin{lemma}\label{Lem:HP}
 Suppose that the contraction in \eqref{Eqt:contractionDatum} is increasingly filtered and that the increasing filtrations on $P$ and $T$ are exhaustive and bounded below. Given a perturbation $\varrho$ of the differential $\delta$ on $P$ satisfying $\varrho(\mathcal{F}^{n}P) \subseteq \mathcal{F}^{n-1}P$, one obtains a new filtered contraction
\[
	\begin{tikzcd}
	(P, \delta  +\varrho) \arrow[loop left, distance=2em, start anchor={[yshift=-1ex]west}, end anchor={[yshift=1ex]west}]{}{h_\flat } \arrow[r,yshift = 0.7ex, "\phi_\flat"] & (T ,  d+\vartheta), \arrow[l,yshift = -0.7ex, "\psi_\flat"]
	\end{tikzcd}
\]
where $\vartheta, \phi_\flat, \psi_\flat, h_\flat$ are defined in~\eqref{Eq: Def of thetaandphi} and~\eqref{Eq: Def of psiandh}.
\end{lemma}

Let $\mathcal{R}$ be a commutative dg algebra. There is a standard construction, called \emph{tensor trick}, on tensor products of contractions of dg $\mathcal{R}$-modules:
\begin{lemma}\label{Lem: Tensor trick}
Given a contraction of dg $\mathcal{R}$-modules
  \[
\begin{tikzcd}
M \arrow[loop left, distance=2em, start anchor={[yshift=-1ex]west}, end anchor={[yshift=1ex]west}]{}{h } \arrow[r,yshift = 0.7ex, "\phi"] & N \arrow[l,yshift = -0.7ex, "\psi"],
\end{tikzcd}
\]
there exists a new contraction on the corresponding reduced tensor (co)algebras
\[
\begin{tikzcd}
T(M) \arrow[loop left, distance=2em, start anchor={[yshift=-1ex]west}, end anchor={[yshift=1ex]west}]{}{Th} \arrow[r,yshift = 0.7ex, "T\phi"] & T(N) \arrow[l,yshift = -0.7ex, "T\psi"],
\end{tikzcd}
\]
where $T(M) = \oplus_{n\geq 1}\otimes_{\mathcal{R}}^nM$ and $T(N) = \oplus_{n\geq 1}\otimes_{\mathcal{R}}^n N$ are reduced tensor (co)algebras of $M$ and $N$, respectively, and
\begin{align*}
  T\phi &= \sum_{n\geq 1} \phi^{\otimes_{\mathcal{R}} n}, & T\psi &= \sum_{n\geq 1}\psi^{\otimes_{\mathcal{R}} n}, & Th &= \sum_n T^n h = \sum_n\sum_{i=1}^{n}(\psi \phi)^{\otimes_{\mathcal{R}}(i-1)} \otimes_{\mathcal{R}} h \otimes_{\mathcal{R}} \id_M^{\otimes_{\mathcal{R}} (n-i)}.
\end{align*}
\end{lemma}
The proof is a straightforward adaptation of Manetti's argument in~\cite{Manetti}, where $\mathcal{R}$ is an ordinary commutative algebra, and thus is omitted. See also~\cite{Berglund}.

\end{document}